\documentclass[12pt]{amsart}
\usepackage{tikz-cd}
\usepackage{enumitem}
\setlist[description]{leftmargin=\parindent,labelindent=\parindent}
\usetikzlibrary{matrix,arrows,decorations.pathmorphing}
\usepackage{amssymb}
\usepackage{amsmath,amscd}
\usepackage{mathrsfs}
\usepackage{url}
\usepackage{cite}
\usepackage{fullpage}
\usepackage{hyperref}
\usepackage{verbatim}
\usepackage{comment}
\usepackage{fancyvrb}
\usepackage{fvextra}
\newtheorem{thm}{Theorem}[section]
\newtheorem{prop}[thm]{Proposition}
\newtheorem{lem}[thm]{Lemma}
\newtheorem{cor}[thm]{Corollary}

\theoremstyle{definition}
\newtheorem{definition}[thm]{Definition}

\newtheorem{rem}[thm]{Remark}

\numberwithin{equation}{section}

\newcommand{\X}{\mathcal{X}}

\newcommand{\T}{\mathcal{T}}
\newcommand{\V}{\mathcal{V}}

\newcommand{\N}{\mathcal{N}}
\newcommand{\B}{\mathcal{B}}

\newcommand{\zz}{\mathbb{Z}}
\newcommand{\cc}{\mathbb{C}}
\newcommand{\qq}{\mathbb{Q}}

\newcommand{\I}{\mathcal{I}}
\newcommand{\C}{\mathcal{C}}
\newcommand{\M}{\mathcal{M}}
\newcommand{\p}{\mathbb{P}}
\newcommand{\pp}{\mathbb{P}}

\renewcommand{\H}{\mathcal{H}}
\newcommand{\F}{\mathcal{F}}

\newcommand{\E}{\mathcal{E}}
\renewcommand{\O}{\mathcal{O}}

\newcommand{\Mg}{\mathcal{M}_g}
\renewcommand{\tilde}{\widetilde}

\DeclareMathOperator{\GL}{GL}

\DeclareMathOperator{\Spec}{Spec}

\DeclareMathOperator{\Sym}{Sym}

\DeclareMathOperator{\PGL}{PGL}
\DeclareMathOperator{\PU}{PU}

\DeclareMathOperator{\BGL}{BGL}
\DeclareMathOperator{\BPU}{BPU}

\newcommand{\W}{\mathcal{W}}
\renewcommand{\gg}{\mathbb{G}}

\renewcommand{\eta}{\sigma}

\newcommand{\Mb}{\overline{\M}}

\title{The rational Chow rings of moduli spaces of hyperelliptic curves with marked points}
\author{Samir Canning}
\author{Hannah Larson}
\thanks{During the preparation of this article, S.C. was partially supported by NSF RTG grant DMS-1502651 and by the Swedish Research Council under grant no. 2016-06596 while in residence at Institut Mittag-Leffler in Djursholm, Sweden during the fall of 2021. H.L. was supported by the Hertz Foundation and NSF GRFP under grant DGE-1656518. This research was partially conducted during the period H.L served as a Clay Research Fellow.}

\begin{document}
\begin{abstract}
We determine the rational Chow ring of the moduli space $\H_{g,n}$ of $n$-pointed smooth hyperelliptic curves of genus $g$ when $n \leq 2g+6$. We also show that the Chow ring of the partial compactification $\I_{g,n}$, parametrizing $n$-pointed irreducible nodal hyperelliptic curves, is generated by tautological divisors. Along the way, we improve Casnati's result that $\H_{g,n}$ is rational for $n \leq 2g+8$ to show $\H_{g,n}$ is rational for $n \leq 3g+5$.
\end{abstract}
\maketitle

\section{Introduction}
The intersection theory of the moduli space of genus $g$ curves $\M_g$ is of central interest in algebraic geometry. The Chow ring with rational coefficients $A^*(\Mg)$ is completely understood for $g\leq 9$ \cite{FaberI, FaberII, Izadi, PenevVakil, 789}. On the other hand, much less is known about the Chow rings of the moduli spaces $\M_{g,n}$ of genus $g$ curves with $n>0$ marked points. The complete picture is only understood for $A^*(\M_{0,n})$ \cite{Keel}, $A^*(\M_{1,n})$ for $n\leq 10$ \cite{Belorousski}, and $A^*(\M_{2,1})$ \cite{FaberPhD}. Because of the structure of the boundary of the compactification $\Mb_g$, computing Chow rings $A^*(\M_{g',n})$ with $g' \leq g$ and $n \leq 2(g - g')$ is a fundamental first step towards understanding $A^*(\Mb_{g})$, which is only completely understood in genus $2$ and $3$ \cite{Mumford,FaberI}.

Recent progress in the unpointed case has been based off of studying Hurwitz spaces parametrizing curves of low gonality and arbitrary genus \cite{Hurwitz}. In this paper, we begin to pursue the same strategy in the pointed case by studying moduli spaces of pointed hyperelliptic curves $\H_{g,n}\subset \M_{g,n}$. The case $n=0$ is straightforward. The coarse space of $\H_g$ is a quotient of an open subset of affine space by a finite group action. Thus, with rational coefficients, one has $A^*(\H_g) = \qq$. (With \emph{integral coefficients}, the Chow ring of $\H_g$ is \emph{not} trivial, and was determined in \cite{DiLorenzo} for $g$ even, and \cite{EdidinFulghesu} for $g$ odd. In this paper, we shall work with rational coefficients throughout.)


When one adds in marked points, however, the geometry and intersection theory of $\H_{g,n}$ becomes more interesting. With regards to its birational geometry, the moduli space $\H_{g,n}$ is known to be (see Figure \ref{f1}, left)
\begin{itemize}
    \item rational when $n \leq 2g+8$ (Casnati \cite{Casnati}),
    \item uniruled when $n \leq 4g+5$ (Benzo and Agostini--Barros \cite{Benzo,AB}),
    \item of Kodaira dimension $4g+3$ when $n = 4g+6$ (Barros--Mullane \cite{BarrosMullane}),
    \item of general type when $n \geq 4g+7$ (Schwarz \cite{Schwarz}).
\end{itemize}
We use the birational geometry of $\H_{g,n}$ as a proxy for its complexity.
In particular, the nice sorts of presentations that are typically used in calculating Chow rings often show a space is (uni)rational. Thus, we might expect that the intersection theory of $\H_{g,n}$ becomes more complicated as $n$ grows, being understandable in the rational range, but possibly quite difficult to access for $n$ large. 

The Chow ring is usually more difficult to compute than the birational type.
Indeed, to prove rationality, one must understand a dense open subset of the space, whereas to determine the Chow ring, one must understand a full stratification and how the pieces fit together.
Previous work on the intersection theory of $\H_{g,n}$ has determined the Chow group in codimension 1 (Scavia \cite{Scavia}) and the full (integral) Chow ring  in the case of 1 marked point (Pernice \cite{Pernice}).
Here, we determine the full rational Chow ring $A^*(\H_{g,n})$ for $n \leq 2g + 6$. 

The picture on the left below summarizes the previously known results about $\H_{g,n}$; the version on the right adds in our new contributions (Corollary \ref{hcor} and Theorem \ref{rat}).

\begin{figure}[h!]
    \centering
    \includegraphics[width=5in]{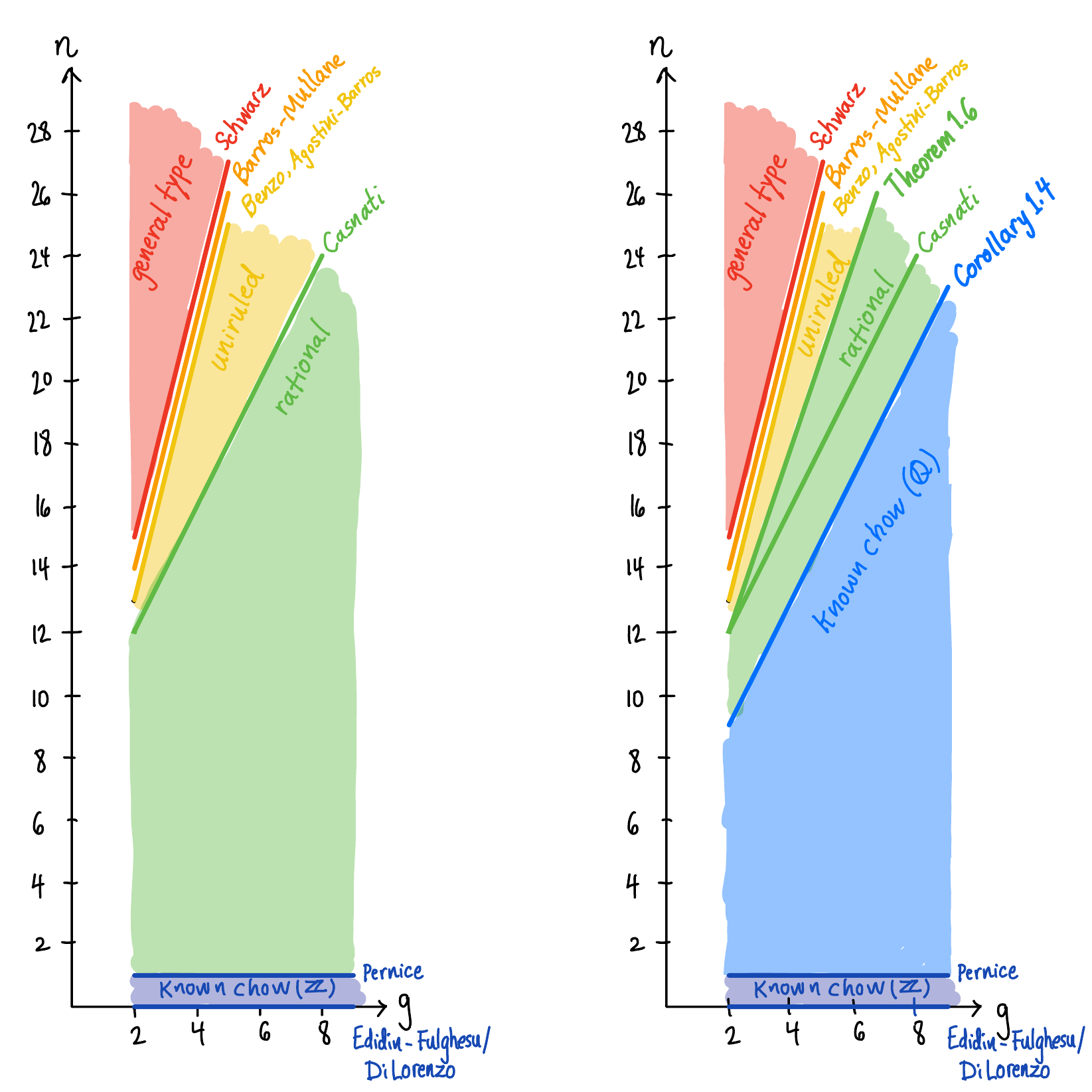}
    \caption{Summary of previously known results about $\H_{g,n}$ (left) and our new results in context (right). Loosely speaking, later colors in the rainbow indicate a ``more explicit" or ``more complete" understanding of the space.}
    \label{f1}
\end{figure}

\subsection{Statement of results}
The main part of our work is to establish that the Chow ring of $\H_{g,n}$ is generated by divisors when $n \leq 2g+6$. Our techniques for doing so naturally extend over a partial compactification of the moduli space. 
Let $\I_{g,n}$ be the stack parametrizing irreducible, nodal $n$-pointed hyperelliptic curves. Equivalently, if $\M_{g,n}^{\mathrm{irr}} \subset \Mb_{g,n}$ denotes the locus of irreducible curves, then $\I_{g,n}$ is the closure of $\H_{g,n}$ in $\M_{g,n}^{\mathrm{irr}}$.
Let $\delta \in A^1(\I_{g,n})$ be the class of the locus of singular curves, and let $\psi_i \in A^1(\I_{g,n})$ denote the restriction of the $i^{\text{th}}$ psi class on $\Mb_{g,n}$ to $\I_{g, n} \subset \Mb_{g,n}$.  We work over an arbitrary algebraically closed field of characteristic not $2$. 

\begin{thm} \label{divgen}
 Let $g \geq 2$ and $n \leq 2g+6$. Then $A^*(\I_{g,n})$ is generated by the divisor classes $\psi_1, \ldots, \psi_n$ and $\delta$.
\end{thm}

\begin{rem}
Theorem \ref{divgen} is used in our forthcoming work \cite{forthcoming}, which makes significant progress towards determining the pairs $(g, n)$ for which $A^*(\Mb_{g,n})$ is tautological. When $n \leq 2g + 6$, Theorem \ref{divgen} guarantees that the failure of $A^*(\Mb_{g,n})$ to be tautological will not be the fault of classes supported on $\I_{g,n} \subset \Mb_{g,n}$. For this application, it is advantageous to know the result for the larger locus $\I_{g,n}$, instead of just $\H_{g,n}$.
\end{rem}

On the smooth locus $\H_{g,n} \subset \I_{g,n}$, it is not hard to determine all relations among the psi classes for any $n$. 
The tautological ring of $\M_{g,n}$ is generated by the psi classes and pullbacks of kappa classes from $\M_g$.
Since the kappa classes restrict to zero on $\H_g \subset \M_g$, it follows that the psi classes generate the tautological ring $R^*(\H_{g,n}) \subseteq A^*(\H_{g,n})$. 
The following proposition shows that the structure of the tautological ring is quite simple.

\begin{prop} \label{rprop}
For $g \geq 2$  and any $n$, we have
\[R^*(\H_{g,n}) = \qq[\psi_1, \ldots, \psi_n]/(\psi_1, \ldots, \psi_n)^2.\] 
\end{prop}

\begin{rem}
In \cite{Tavakol}, Tavakol determines the tautological ring of a different partial compactification $\H_{g,n}^{\mathrm{rt}} \supset \H_{g,n}$, which implies Proposition \ref{rprop} by excision. In Lemma \ref{r2}, we nearly determine the tautological ring of our partial compactification $\I_{g,n}$, which also immediately implies Proposition \ref{rprop}.
\end{rem}

Combining Proposition \ref{rprop} with Theorem \ref{divgen}, we obtain the following.
\begin{cor} \label{hcor}
If $n \leq 2g + 6$, then 
\[A^*(\H_{g,n}) = R^*(\H_{g,n}) = \qq[\psi_1,...,\psi_n]/(\psi_1,\ldots,\psi_n)^2.\]
\end{cor}

\begin{rem}
Note that Corollary \ref{hcor} and Theorem \ref{divgen} do \emph{not} hold for all $n$. In \cite[Theorem 3]{GraberPandharipande}, Graber and Pandharipande prove that $A^*(\M_{2,20}) \neq R^*(\M_{2,20})$ by producing an explicit non-tautological algebraic cycle. (Note that $n = 20$ falls in the range where $\M_{2,n} = \H_{2,n}$ is of general type.)
\end{rem}

Our approach to proving Theorem \ref{divgen} is inspired by Casnati's method of constructing models in $\pp^2$ \cite{Casnati}, but we use models on $\pp^1 \times \pp^1$ instead. One consequence of our approach is an improvement on Casnati's bound for when these spaces are rational.

\begin{thm} \label{rat}
If $n \leq 3g+6$, then $\H_{g,n}$ is rational.
\end{thm}


The discrepancy between the range where we can prove $\H_{g,n}$ is rational
and where we can determine the Chow ring is the difference between when a \emph{general} collection of $n$ points on $\pp^1 \times \pp^1$ impose independent conditions on a certain linear system versus when \emph{every} configuration of $n$ points we care about imposes independent conditions.

\subsection{Structure of the paper}
In Section \ref{stacks}, we define stacks $\I_{g,n}$ of irreducible, nodal pointed hyperelliptic curves and the locally closed substacks in the stratification we shall use.
In Section \ref{tsec}, we define the tautological classes on $\I_{g,n}$ and prove several relations among them. This involves constructing some explicit quotient stacks with the same coarse spaces as $\I_{g,0}$ and $\I_{g,1}$. For $n \geq 2$ however, we cannot give a global quotient description of $\I_{g,n}$. Instead, we build each of the pieces of our stratification using models on $\pp^1 \times \pp^1$ in Section \ref{quotstack}.
The two maps to $\pp^1$ in forming these models are the hyperelliptic map and the complete linear series of a (weighted) sum of marked points having degree $g+1$. For this linear series to be base point free, certain points are prohibited from being conjugate or Weierstrass. 
An inductive argument allows us to eliminate strata where a pair of points is conjugate. 
Meanwhile, we construct strata with marked Weierstrass points by imposing tangency conditions to vertical lines of the ruling in our models.
At the end of Section \ref{quotstack}, we present a proof of Theorem \ref{divgen}, relying on Scavia's result over $\mathbb{C}$. In Section \ref{pp}, we work to establish the needed parts of Scavia's theorem over an arbitrary algebraically closed field of characteristic not $2$.

\subsection{Notations and Conventions}
If the ground field is not mentioned, we are working over an algebraically closed field $k$ of characteristic not $2$. We will explicitly mention where we are using results that are only known to hold over $\mathbb{C}$. We use the classical subspace convention for projective bundles.

\subsection*{Acknowledgments} We are grateful to our advisors, Elham Izadi and Ravi Vakil, respectively, for the many helpful conversations. We thank Dan Petersen for his comments and for pointing out \cite{Tavakol}. We also thank Renzo Cavalieri for his comments on an earlier draft.

\section{Pointed irreducible, nodal hyperelliptic curves}\label{stacks}
Let $S$ be a $k$-scheme. A family of irreducible, nodal hyperelliptic curves of genus $g$ over $S$ is a morphism of $S$ schemes $C\rightarrow P\rightarrow S$ where $C\rightarrow S$ is a family of irreducible, nodal curves of genus $g$, $P\rightarrow S$ is a $\pp^1$-fibration, and $C\rightarrow P$ is finite and flat of degree $2$. An $n$-pointed family of irreducible, nodal hyperelliptic curves over $S$ is a family of irreducible, nodal hyperelliptic curves $C\rightarrow P\rightarrow S$ together with $n$ disjoint sections $p_1,\dots,p_n:S\rightarrow C$ of $C\rightarrow S$ such that the sections are disjoint from the nodes in fibers of $C \to S$. 
An arrow between families $(C\rightarrow P\rightarrow S, p_1,\dots,p_n:S\rightarrow C)$ and $(C'\rightarrow P'\rightarrow S',p_1',\dots,p_n':S\rightarrow C)$ is simply a commutative diagram
\[
\begin{tikzcd}
C \arrow[d] \arrow[r] & P \arrow[d] \arrow[r] & S \arrow[d] \\
C' \arrow[r]          & P' \arrow[r]          & S'         
\end{tikzcd}
\]
that also commutes the sections.
For brevity, we sometimes omit $S$ and $P$, and write $(C,p_1,\dots,p_n)$ for a family of $n$-pointed irreducible, nodal hyperelliptic curves.
\begin{definition}
The stack of \emph{$n$-pointed irreducible, nodal genus $g$ hyperelliptic curves} $\I_{g,n}$ is the stack whose objects are families of $n$-pointed irreducible, nodal hyperelliptic curves with morphisms defined as above.

The stack of \emph{$n$-pointed genus $g$ hyperelliptic curves} $\H_{g,n} \subset \I_{g,n}$ is the substack defined by the additional condition that $C \to S$ is smooth.
\end{definition}

Let $\M_{g,n}^{\mathrm{irr}} \subset \M_{g,n}$ denote the substack of pointed, irreducible curves.
One can also describe $\I_{g,n}$ as follows.

\begin{lem}
$\I_{g,n}$ is the closure in $\M_{g,n}^{\mathrm{irr}}$ of $\H_{g,n} \subset \M_{g,n}$.
\end{lem}
\begin{proof}
The closure of $\H_{g,n}$ inside $\Mb_{g,n}$ is the preimage along $\Mb_{g,n} \to \Mb_g$ of
$\overline{\H}_g \subset \Mb_g$.
The locus $\overline{\H}_g \subset \Mb_g$ is well-understood as the image of the space of admissible degree $2$ covers.
The admissible covers with irreducible source are exactly the irreducible nodal hyperelliptic curves. When we add markings, we require that they do not meet the node, so the curve stays irreducible.
\end{proof}

\subsection{The hyperelliptic involution and the dualizing sheaf} \label{hi}
Given an irreducible, nodal hyperelliptic curve $C$, let $\nu: \tilde{C} \to C$ be the normalization. Write $q_i, q_i'$, $i = 1, \ldots, m$ for the pairs of points lying over the nodes of $C$.
Considering the composition $\tilde{C} \to C \to \pp^1$, we see that $\tilde{C}$ admits a degree $2$ map $\tilde{C} \to \pp^1$ and each pair $q_i, q_i'$ lies in the same fiber of this map.
\begin{center}
\includegraphics[width=2.5in]{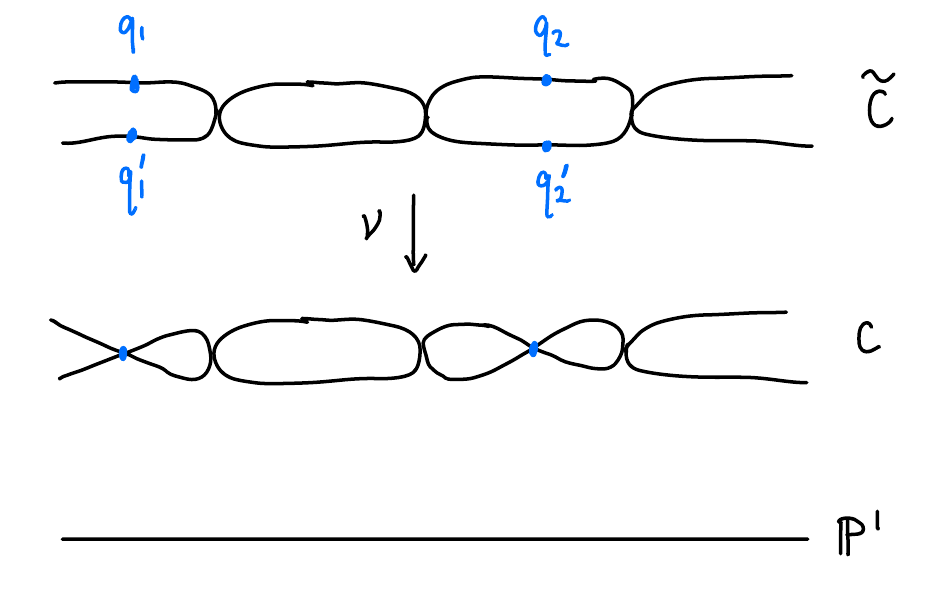}
\end{center}
If $g(\tilde{C}) \geq 2$, then $\tilde{C}$ is a hyperelliptic curve and $q_i, q_i'$ are conjugate under the hyperelliptic involution on $\tilde{C}$.
If $g(\tilde{C}) = 1$, then we have the condition $q_i + q_i' \sim q_j + q_j'$ for all $i, j$. If $g(\tilde{C}) = 0$, then every pair of points are linearly equivalent, but the following condition on triples $(i, j, k)$ must be satisfied:
the degree $2$ polynomial vanishing at $q_i+q_i'$ must lie in the span of the degree $2$ polynomial vanishing at $q_j+q_j'$ and the degree $2$ polynomial vanishing at $q_k + q_k'$.

In summary, irreducible nodal hyperelliptic curves are just built by gluing together conjugate points on a hyperelliptic curve $\tilde{C}$ of lower genus (with a suitable modification when $\tilde{C}$ has genus $1$ or $0$).

\begin{lem}
Let $C$ be an irreducible, nodal hyperelliptic curve of genus $g \geq 2$. Then $C$ admits a unique degree $2$ map $\alpha: C \to \pp^1$ (up to automorphisms of $\pp^1$). 
\end{lem}
\begin{proof}
In this proof, we write that two maps to $\pp^1$ are equal if they differ by composition by an automorphism of $\pp^1$.
Let $\nu: \tilde{C} \to C$ be the normalization.
Suppose $\alpha': C \to \pp^1$ is another degree $2$ map. Then $\alpha' \circ \nu$ and $\alpha' \circ \nu$ are both degree $2$ maps $\tilde{C} \to \pp^1$. If $g(\tilde{C}) \geq 2$, then we immediately see $\alpha' \circ \nu = \alpha \circ \nu$, and hence $\alpha' = \alpha$.

If $g(\tilde{C}) = 1$, then there is a unique degree two map $\tilde{C} \to \pp^1$ that sends $q_1$ and $q_1'$ to the same point, namely the complete linear system of $\O_{\tilde{C}}(q_1 + q_1')$. As $\alpha' \circ \nu$ and $\alpha \circ \nu$ both send $q_1$ and $q_1'$ to the same point, we see $\alpha' \circ \nu = \alpha \circ \nu$, and hence $\alpha' = \alpha$.

If $g(\tilde{C}) = 0$, then since $g(C) \geq 2$, we know $C$ has at least two nodes. But there is a unique degree two map $\tilde{C} \to \pp^1$ which sends $q_1$ and $q_1'$ to the same point \emph{and} sends $q_2$ and $q_2'$ to the same point. Hence, again we see $\alpha' \circ \nu = \alpha \circ \nu$, and so $\alpha' = \alpha$.
\end{proof}

\begin{definition}
We call the unique degree two map $\alpha: C \to \pp^1$ the hyperelliptic map. We write $L := \alpha^*\O_{\pp^1}(1)$ for the corresponding degree $2$ line bundle.
Given a point $p \in C$, we write $\overline{p}$ for its conjugate under the hyperelliptic involution. Nodes are always fixed under the hyperelliptic involution. 
\end{definition}



We now note that two facts, which are well-known for smooth hyperelliptic curves, also hold in the irreducible nodal case (by the same arguments as in the smooth case).

\begin{lem}
Let $C$ be an irreducible, nodal hyperelliptic curve of genus $g$, $\alpha: C \to \pp^1$ be the hyperelliptic map, and let $L = \alpha^*\O_{\pp^1}(1)$ be the corresponding degree $2$ line bundle. Then the dualizing sheaf satisfies $\omega_C \cong L^{\otimes g-1}$.
\end{lem}
\begin{proof}
By Riemann--Roch,
\[h^0(C, L^{\otimes g-1}) - h^0(C, \omega_C \otimes (L^{\otimes g-1})^{\vee}) = (2g - 2) - g + 1 = g - 1. \]
Meanwhile, we have $h^0(C, L^{\otimes g-1}) \geq h^0(\pp^1, \O(g-1)) = g$, so $h^0(C, \omega_C \otimes (L^{\otimes g - 1})^\vee) \geq 1$. But $ \omega_C \otimes (L^{\otimes g - 1})^\vee$ has degree $0$, so it has a non-trivial section if and only 
if it trivial, which means $\omega_C \cong L^{\otimes g - 1}$.
\end{proof}

Geometrically, this tells us that the complete linear system for $\omega_C$ sends $C$ to a double cover of a degree $g - 1$ rational normal curve in $\pp^{g-1}$. This implies the following.

\begin{lem} \label{georr}
Suppose $D = x_1 + \ldots + x_g$ is an effective degree $g$ divisor with $h^0(C, \O(D)) \geq 2$. Then $x_i = \overline{x}_j$ for some $i \neq j$.
\end{lem}
\begin{proof}
By Riemann--Roch, if $h^0(C, \O(D)) \geq 2$, then $h^0(C, \omega_C(-D)) \geq 1$, which is to say the image of $D$ under the canonical embedding $C \to \pp^{g-1}$ lies in a hyperplane.
However, any degree $g$ divisor on a rational normal curve in $\pp^{g-1}$ spans 
all of $\pp^{g-1}$. Therefore, a length two subscheme of $D$ must be sent to a length one subscheme under the canonical map, which is to say $x_i = \overline{x}_j$ for some $i \neq j$.
\end{proof}

\subsection{Our stratification}
In order to compute the Chow ring of $\I_{g,n}$, we will introduce a stratification. It keeps track of the two interesting geometric phenomena that happen on pointed hyperelliptic curves:
\begin{enumerate}
    \item a marked point is fixed by the hyperelliptic involution (this is called a Weierstrass point), or
    \item two marked points are conjugate under the hyperelliptic involution.
\end{enumerate} 
Recall that, if they exist, the nodes are not allowed to be marked. Recall that we write $\overline{p}$ for the conjugate of $p$ under the hyperelliptic involution.

First, define divisors
\[D_{ij} =\{(C, p_1, \ldots, p_n) \in \I_{g,n} : p_i = \overline{p}_j\}, \]
so $D_{ij}$ is the locus where $p_i$ and $p_j$ are conjugate under the hyperelliptic involution, and $D_{ii}$ is the locus where $p_i$ is a Weierstrass point.

\begin{prop}\label{induction}
Let $i\neq j$ and suppose that $A^*(\I_{g,n-1})$ is generated by divisors. Then the Chow ring of $D_{ij}\subset \I_{g,n}$ is generated by restrictions of divisors from $\I_{g,n}$.
\end{prop}
\begin{proof}
There is a commutative square
\[
\begin{tikzcd}
D_{ij} \arrow[hook]{r} \arrow{d}[swap]{\sim} & \I_{g,n} \arrow{d} \\
\I_{g,n-1} \smallsetminus D_{ii} \arrow{r} & \I_{g,n-1}, \\
\end{tikzcd}
\]
where the vertical maps forget the $j^{\mathrm{th}}$ marked point. The left vertical map is an isomorphism. Note that above $D_{ij}\subset \I_{g,n}$, while $D_{ii}\subset \I_{g,n-1}$.
Taking Chow rings, we have
\[
A^*(\I_{g,n-1})\twoheadrightarrow A^*(\I_{g,n-1}\smallsetminus D_{ii})\cong A^*(D_{ij}).
\]
By the assumption that $A^*(\I_{g,n-1})$ is generated by divisors,  it follows that $A^*(D_{ij})$ is generated by restrictions of divisors from $\I_{g,n}$.
\end{proof}

\begin{rem}
A similar idea to the above --- relating divisors where a certain linear equivalence holds to an open subset of the moduli space with one less point --- was used by Belorousski \cite{Belorousski} in studying the Chow rings of $\M_{1,n}$ for $n \leq 10$.
\end{rem}

Thus, if $A^*(\I_{g,n-1})$ is generated by divisors, then using the push-pull formula, every class supported on $D_{ij} \subset \I_{g,n}$ for $i\neq j$ is a polynomial in divisor classes. Let us write
\[\I_{g,n}^\circ := \I_{g,n} \smallsetminus \bigcup_{i < j} D_{ij}.\]
On $\I_{g,n}^\circ$, no pair of marked points is conjugate, but the marked points may still be Weierstrass.
Thus, by excision and Proposition \ref{induction}, to establish that $A^*(\I_{g,n})$ is generated by divisors, it will suffice to show that $A^*(\I_{g,n}^\circ)$ is generated by divisors.
When $g +1 \leq n \leq 2g+6$ we construct $\I_{g,n}^\circ$ directly as quotient of an open subset of a projective bundle over an open subset of affine space.

Meanwhile, for $n \leq g$, we must further stratify $\I_{g,n}^\circ$.
To set this up, we will also require the slightly smaller open where no pair of points are conjugate and the $i^{\mathrm{th}}$ point is prohibited from being Weierstrass:
\begin{align}\label{Hl}
\I_{g,n}^{\circ,i} &:=\I_{g,n} \smallsetminus \left( D_{ii} \cup \bigcup_{j<k} D_{jk} \right) &\qquad &(\text{if $n < g+1$})
\intertext{In order to combine arguments in the cases $n < g+1$ and $n \geq g+1$, we use the convention that}
\I_{g,n}^{\circ, i} &:= \I_{g,n}^\circ &\qquad  &(\text{if $n \geq g+1$}). \label{conv}
\end{align}

The basic geometric loci left on $\I_{g,n}^\circ$ are the loci where some collection of points are Weierstrass. When $n < g+1$, we are going to stratify $\I_{g,n}^\circ$ into disjoint locally closed strata $W_0 \cup W_1 \cup \cdots \cup W_n$.
The subscript $i$ will measure the number of initial consecutive Weierstrass points when the points are read in order.
 Precisely, $W_i$ is the locally closed stratum
\begin{align}
W_i &:= \{(C, p_1, \ldots, p_n) \in \I_{g,n}^\circ: p_j = \overline{p}_j \text{ for $j \leq i$ and } p_{i+1} \neq \overline{p}_{i+1}\} \label{si}\\
&= \I_{g,n}^\circ \cap \left( \bigcap_{j \leq i} D_{jj} \smallsetminus \bigcap_{j \leq i+1} D_{jj}\right). \label{si2}
\end{align}
Note that $W_i$ is a closed subset of $\I_{g,n}^{\circ,i+1}$.
In particular, $W_0 = \I_{g,n}^{\circ, 1}$.
Generically, curves in $W_i$ have $i$ marked Weierstrass points, so $W_i$ has codimension $i$.
For each $i$, we have $\overline{W}_i = W_i \cup \cdots \cup W_n$.

\begin{lem} \label{fc}
The fundamental class $[\overline{W}_i]$ is a product of divisors.
\end{lem} 
\begin{proof}
By \eqref{si2}, we see that $\overline{W}_i = \I_{g,n}^\circ \cap(\bigcap_{j \leq i} D_{jj})$ is a dimensionally transverse intersection of $i$ divisors.
\end{proof}

Thus, to show that $A^*(\I_{g,n}^\circ)$ is generated by divisors, it will suffice to 
show that each $A^*(W_i)$ is generated by restrictions of divisors from $\I_{g,n}^\circ$.
We shall show this by constructing $\I_{g,n}^{\circ,i+1}$ as a quotient of an open subset of a projective bundle over an open subset of affine space. Furthermore, $W_i \subset \I_{g,n}^{\circ,i+1}$ will be the quotient of an open subset of the projectivization of a particular subbundle. This will allow us to see that $A^*(W_i)$ is generated by restrictions of classes from $A^*(\I_{g,n}^{\circ, i+1})$, which we in turn show is generated by divisors (note that such divisors are restrictions from $\I_{g,n}^\circ$ by excision).

The case $W_n$ is slightly different and will be treated in Lemma \ref{Wn}.

\section{Tautological classes and relations} \label{tsec}
We define the \emph{tautological ring} $R^*(\I_{g,n}) \subseteq A^*(\I_{g,n})$ (resp. $R^*(\H_{g,n}) \subseteq A^*(\H_{g,n})$) to be the subring generated restrictions of tautological classes from $\Mb_{g,n}$. 

\subsection{The psi and boundary divisors}
Let 
\[
\pi:\C_{g,n}\rightarrow \I_{g,n}
\]
be the universal family of curves over $\I_{g,n}$. The morphism $\pi$ has $n$ universal pairwise disjoint sections
\[
p_1,\dots,p_n:\I_{g,n}\rightarrow \C_{g,n},
\]
which are also disjoint from the singular locus of $\C_{g,n} \to \I_{g,n}$.
\begin{definition}
The cotangent line bundles are the bundles
\[
\mathbb{L}_i:=p_i^*\omega_{\pi}.
\]
The $\psi$ classes are the divisors
\[
\psi_i:=c_1(\mathbb{L}_i)\in A^1(\I_{g,n}).
\]
\end{definition}
The $\psi$ classes behave well with respect to pulling back by the morphisms that forget marked points. 
There are natural maps $\pi_j: \I_{g,n} \to \I_{g,n-1}$ where we forget the $j^{\mathrm{th}}$ marked point, under which we have $\pi_j^* \psi_i = \psi_i$ (this equality uses that our curve is irreducible and the markings do not meet the singular locus).

\begin{definition}
Let $\Delta \subset \I_{g,n}$ be the divisor of singular curves, so $\I_{g,n} \smallsetminus \Delta = \H_{g,n}$. Define 
\[\delta :=[\Delta] \in A^1(\I_{g,n}).\]
\end{definition}

\begin{lem}
The divisor $\Delta \subset \I_{g,n}$ is irreducible.
\end{lem}
\begin{proof}
By the discussion in Section \ref{hi}, we see that $\Delta$ is the image of $D_{12} \subset \I_{g,n+2}$ under the map $D_{12} \to \I_{g,n}$ that glues together the first two marked points, which are conjugate.
As in Lemma \ref{induction}, we have $D_{12}$ isomorphic to $\I_{g,n+1} \smallsetminus D_{11}$, which is irreducible.
\end{proof}

In \cite[Theorem 1.1]{Scavia}, Scavia shows that, over the complex numbers, the $\psi$ classes and boundary divisors form a basis for the Picard group of $\overline{\H}_{g,n}$. Our $\I_{g,n}$ is the complement of all boundary divisors in $\overline{\H}_{g,n}$ besides the divisor $\Delta$ of irreducible nodal curves. Because $\Delta$ is irreducible, we obtain the following.

\begin{thm} \label{Scavia}
Let $k= \cc$ and take $\I_{g,n}$ as a stack over $\Spec \cc$. For every $g\geq 2$ and $n\geq 0$, $A^1(\I_{g,n})$ has a basis given by $\psi_1,\dots,\psi_n$ and $\delta$.
\end{thm}

We will prove
Theorem \ref{divgen} for $\I_{g,n}$ defined over any algebraically closed field $k$ of characteristic not $2$. Using Scavia's result (Theorem \ref{Scavia}), we can simplify the proof when we work over the complex numbers: to prove Theorem \ref{divgen} over $\mathbb{C}$ it remains to prove that $A^*(\I_{g,n})$ is generated by divisors. We have organized the paper so that the reader who is only interested in the case over $\mathbb{C}$ need not read the final Section \ref{pp}.

Finally, let us relate the geometric classes in our stratification to the tautological classes we have just defined. In \cite{EdidinHu}, Edidin and Hu use the method of test curves to compute compute the classes of $D_{ij}$ in $A^1(\overline{\H}_{g,n})$. Restricting to $\I_{g,n}$, we have
\begin{align}
    [D_{ij}] &= \frac{1}{g-1}(\psi_i + \psi_j) - \frac{1}{2(2g+1)(g-1)}\delta \label{dij} \\
    [D_{ii}] &= \frac{g+1}{g-1} \cdot \psi_i - \frac{1}{2(2g+1)(g-1)}\delta. \label{dii}
\end{align}
Although Edidin--Hu cite Scavia's result over $\cc$ that the $\psi$ classes and boundary divisors generate the Picard group, their argument does not use the complex numbers in any other way. Thus, by our work in Section \ref{pp} generalizing Scavia's result, these formulas will be shown to hold over any ground field of characteristic not $2$.

\subsection{Relations} \label{relsec} We now calculate several relations among the $\psi$ and $\delta$ classes on $\I_{g,n}$. 
These relations nearly determine $R^*(\I_{g,n})$ and fully determine $R^*(\H_{g,n})$.

First, we compute $A^*(\I_{g,0})$.

\begin{lem} \label{0pt}
We have 
$A^*(\I_{g,0}) = \qq[\delta]/(\delta^3)$.
\end{lem}
\begin{proof}
An irreducible nodal hyperelliptic curve is determined by its branch divisor. The branch divisor is a degree $2g+2$ divisor on $\pp^1$ with no point of multiplicity $3$ or more (to make the double cover at worst nodal) and at least one point of multiplicity $1$ (to make the cover irreducible).
Let $\V$ be the universal rank $2$ vector bundle on $\BGL_2$.
Define $\Delta_3 \subset \Sym^{2g+2} \V^\vee$ to be the degree $2g + 2$ forms on $\pp \V$ with a root of multiplicity $3$ or more. Let $R \subset \Sym^{2g+2} \V^\vee$ be the forms with no simple roots. Note that $R$ has codimension $g+1 \geq 3$.

Then the coarse space of $\I_{g, 0}$ is the same as the coarse space of \[\Sym^{2g+2} \V^\vee \smallsetminus (\Delta_3 \cup R),\]
which we think of as an open subset of a vector bundle over $\BGL_2$. We therefore identify the Chow ring of $\I_{g, 0}$ with the Chow ring of $\Sym^{2g+2} \V^\vee \smallsetminus (\Delta_3 \cup R)$
We can determine the relations that come from excising $\Delta_3$ using a diagram of the form
\begin{center}
\begin{tikzcd}
\tilde{\Delta}_3 \arrow{dr} \arrow{r} & \pi^* \Sym^{2g+2}\V^\vee \arrow{r}\arrow{d} & \Sym^{2g+2}\V^\vee \arrow{d} \\
& \pp \V \arrow{r}{\pi} & \BGL_2,
\end{tikzcd}
\end{center}
where $\tilde{\Delta}_3$ is the kernel of the principal parts evaluation map
\[\pi^*\Sym^{2g+2} \V^\vee = \pi^*\pi_* \O_{\pp \V}(2g+2) \rightarrow P^2_{\pp \V/\BGL_2}(\O_{\pp \V}(2g+2)). \]
The space $\tilde{\Delta}$ parametrizes points on $\pp \V$ together with forms having a root of multiplicity $3$ or more at that point. In particular, $\tilde{\Delta}_3$ surjects onto $\Delta_3 \subset \Sym^{2g+2} \V^\vee$. Since we work with rational coefficients, all classes supported on $\Delta_3$ are pushforwards of classes from $\tilde{\Delta}_3$.

Let $c_1, c_2$ be the generators of $A^*(\BGL_2)$ and let $z = c_1(\O_{\pp V}(1))$. 
By the projective bundle theorem, $A^*(\pp \V) = \zz[c_1,c_2,z]/(z^2 + c_1z + c_2)$.
The fundamental class of $\tilde{\Delta}_3 \subset \pi^*\Sym^{2g+2}\V^\vee$ is the top Chern class of the principal parts bundle $c_3(P^2_{\pp \V/\BGL_2}(\O_{\pp \V}(2g+2)))$.
To calculate this top Chern class, we use that $P^2_{\pp \V/\BGL_2}(\O_{\pp \V}(2g+2))$ is filtered by the line bundles
\[\O_{\pp \V}(2g+2), \qquad \O_{\pp \V}(2g+2) \otimes \Omega_{\pp \V/\BGL_2}, \qquad  \O_{\pp \V}(2g+2) \otimes \Omega_{\pp \V/\BGL_2}^{\otimes 2}\]
and the splitting principle.
By the relative Euler sequence, $c_1(\Omega_{\pp \V/\BGL_2}) = -2z -c_1$.
By \cite[Trapezoid Lemma 2.1]{part2}, the image of $A^{*-1}(\tilde{\Delta}_3) \to A^*(\Sym^{2g+2}\V^\vee)$ is the ideal in $A^*(\Sym^{2g+2}\V^\vee) \cong A^*(\BGL_2)$ generated by
\begin{align}
\pi_*(c_3(P^2_{\pp \V/\BGL_2}(\O_{\pp \V}(2g+2))) ) &=(8g^3 + 12g^2 + 4g)c_1^2 + (-8g^3 + 8g)c_2 \label{firstp} \\
\intertext{and}
c_3(P^2_{\pp \V/\BGL_2}(\O_{\pp \V}(2g+2)) \cdot z) &= (-8g^3 - 12g^2 - 4g)c_1^3 + (16g^3 + 12g^2 - 8g - 4)c_1c_2. \label{secondp}
\end{align}
Setting the right-hand side of \eqref{firstp} to zero tells us that $c_2$ is a non-zero multiple of $c_1^2$; then, setting \eqref{secondp} to zero tells us $c_1^3 = 0$.
From this, it follows that $A^*(\I_{g,0}) \cong \qq[c_1]/(c_1^3)$.

Finally, we express the class $\delta$ in terms of $c_1$. The divisor $\Delta \subset \Sym^{2g+2} \V^\vee$ corresponds to the degree $2g+2$ forms on $\pp \V$ with a double root. In a similar fashion to the above argument, we may realize $\Delta$ as the image of the kernel of the principal parts evaluation map
\[\pi^* \Sym^{2g+2} \V^\vee = \pi^*\pi_*\O_{\pp \V}(2g+2) \to P^1_{\pp \V/\BGL_2}(\O_{\pp \V}(2g+2)).\]
Pushing forward, we have
\begin{equation} \label{dclass} \delta = \pi_*(c_2(P^1_{\pp \V/\BGL_2}(\O_{\pp \V}(2g+2) ))) = (-4g^2 - 6g - 2)c_1.
\end{equation}
In particular, $\delta$ is a non-zero multiple of $c_1$,
so we get the claimed presentation of $A^*(\I_{g,0})$.
\end{proof}

\begin{cor} \label{tcor}
The tautological ring $R^*(\I_{g,n})$ is generated by the $\psi$ classes and $\delta$.
\end{cor}
\begin{proof}
By definition, $R^*(\I_{g,n})$ is generated by the restrictions of psi classes, the kappa classes, and boundary strata from $\Mb_{g,n}$. 
Since we require our curves to be irreducible, the only boundary strata that meet $\I_{g,n}$ are irreducible curves with some number of nodes, which are all pulled back from $\Mb_{g}$. 
On $\Mb_{g,n}$, the kappa classes are expressible as polynomials in the pullbacks of kappa classes from $\Mb_g$ and psi classes and boundary classes. Considering the commutative square
\begin{center}
\begin{tikzcd}
\I_{g,n} \arrow{r} \arrow{d} & \Mb_{g,n} \arrow{d} \\
\I_{g,0} \arrow{r} & \Mb_{g}
\end{tikzcd}
\end{center}
and the fact that $A^*(\I_{g,0})$ is generated by $\delta$,
we see that $R^*(\I_{g,n})$ is generated by $\delta$ and the $\psi$ classes.
\end{proof}

Next, we shall compute $A^*(\I_{g,1})$. This determines some relations among $\psi$ and $\delta$.




\begin{lem} \label{1pt}
We have 
\[A^*(\I_{g,1}) = \frac{\qq[\psi_1, \delta]}{(\delta \psi_1 +(2g-1)\delta^2, \psi_1^2 + a_g \delta^2, \delta^3)}\] where $a_g$ is a constant depending on $g$.
\end{lem}
\begin{proof}
In the coarse space of $\I_{g,1}$, the points $(C, p)$ and $(C, \overline{p})$ are identified. This allows us to easily construct another stack with the same coarse space: we need a degree $2g+2$ form $F$ on $\pp \V$ with no roots of multiplicity $3$ or more (and at least one root of multiplicity $1$), together with a point $p \in \pp \V$ which is not a double root of $V(F)$. (The point $p$ should not be a double root because we marked points are not allowed to be singular.)
Let 
\begin{align*}D &= \{(p, F) : p \text{ has multiplicity $\geq 2$ in $V(F)$}\} \\
&\subset \pp \V \times_{\BGL_2} (\Sym^{2g+2} \V^\vee \smallsetminus (\Delta_3 \cup R)) = \pp \V \times_{\BGL_2} \I_{g,0}. 
\end{align*}
Then the coarse space of $\I_{g,1}$ is the same as the coarse space of 
\[(\pp \V \times_{\BGL_2} \I_{g,0}) \smallsetminus D.\]
We have $A^*(\pp \V \times_{\BGL_2} \I_{g,0}) = A^*(\I_{g,0})/(z^2 + c_1z + c_2)$. In terms of the generator $\delta$, we have
\begin{equation} \label{pbtrel}0 = z^2 + c_1z + c_2 = z^2 - \frac{1}{2(2g+1)(g+1)}\delta z - \frac{1}{8(2g+1)(g-1)(g+1)^2}\delta^2, 
\end{equation}
where we have used \eqref{firstp} and \eqref{dclass} to write $c_1$ and $c_2$ in terms of $\delta$.

It remains to find the relations imposed by excising $D$. Let $\pi: \pp \V \to \BGL_2$ be the projection. We have that $D$ is the total space of the kernel of
\[\pi^* \Sym^{2g+2} \V^\vee = \pi^*\pi_* \O_{\pp \V}(2g+2) \to P^1_{\pp \V/\BGL_2}(\O_{\pp \V}(2g+2)).\]
Because $D$ is a vector bundle over $\pp \V$, all classes supported on it are multiples of its fundamental class.
The fundamental class of $D$ is the top Chern class of the principal parts bundle
\begin{align}c_2(P^1_{\pp \V/\BGL_2}(\O_{\pp \V}(2g+2))) &= ((-4g^2 - 6g - 2)c_1)z + (-4g^2 - 4g)c_2 \notag \\
&= \delta z + \frac{g}{2(2g+1)(g-1)(g+1)} \delta^2. \label{rel1}
\end{align}

We now wish to express $z$ in terms of $\psi_1$. The divisor $D_{11} \subset \I_{g,1}$ is
\[D_{11} = \{(p, F) : p \in V(F)\} \subset \pp \V \times_{\BGL_2} \Sym^{2g+2}\V^\vee,\]
which is the total space of the kernel of the regular evaluation map
\[\pi^*\Sym^{2g+2} \V^\vee = \pi^*\pi_*\O_{\pp \V}(2g+2) \to \O_{\pp \V}(2g+2).
\]
Hence, $[D_{11}] = c_1(\O_{\pp \V}(2g+2)) = (2g+2)z$. Using \eqref{dii}, we find
\[z =\frac{1}{2g+2}[D_{11}] = \frac{1}{2g+2}\left(\frac{g+1}{g-1} \psi_1 - \frac{1}{2(2g+1)(g-1)}\delta\right).\]
When we plug this into \eqref{rel1} and divide out by non-zero constants, we find 
\[\delta \psi_1 + (2g-1) \delta^2 = 0.\]
Finally, we plug the formula for $z$ into \eqref{pbtrel} and use the above to eliminate the $\delta \psi_1$ terms. This gives
\[\psi_1^2 + \frac{16g^4-24g^3+16g^2+8g-3}{4(2g+1)^2(g+1)^2}\delta^2 = 0,\]
so we obtain the desired presentation.
\end{proof}

Similar techniques to the previous lemmas allow us find the Chow ring of the stratum $W_n \subset \I_{g,n}$ where all marked points are Weierstrass. We deal with these strata now, since our method of constructing models of curves on $\pp^1 \times \pp^1$ does not work on them.

\begin{lem} \label{Wn}
Suppose $2 \leq n \leq 2g + 2$. The stratum $W_n \subset \I_{g,n}$ has $A^*(W_n) = \qq$.
\end{lem}
\begin{proof}
As before, let $\V$ be the tautological rank $2$ bundle on $\BGL_2$. Let $\eta_i: (\pp \V)^n \to \pp \V$ be projection onto the $i^{\mathrm{th}}$ factor, and write $z_i := \eta_i^* \O_{\pp \V}(1)$. Let $U \subset (\pp \V)^n$ be the complement of all the big diagonals. The locus where $p_i = p_j$ is defined by the vanishing of the composition $\eta_i^* \O_{\pp \V}(-1) \to \V \to  \V/\eta_j^*\O_{\pp \V}(-1)$.
Thus, we obtain relations $z_i + z_j + c_1 = 0 \in A^*(U)$ for each $i \neq j$.

Now consider
\[X = \{(p_1, \ldots, p_n, F): p_i \in V(F)\} \subset U \times_{\BGL_2} \Sym^{2g+2} \V^\vee,\]
which is the total space of a vector bundle over $U$. In particular, $A^*(X) \cong A^*(U),$ which is generated by $c_1, c_2, z_1, \ldots, z_n$.
Since the marked points are not allowed to be singular, we wish to remove from $X$ the divisors $X_i$ where $V(F)$ has multiplicity $2$ or more at $p_i$. Each $X_i \subset X$ is a subbundle of $X$ whose cokernel is $\eta_i^*(\Omega_{\pp V/\BGL_2} \otimes \O_{\pp \V}(2g+2))$. Hence, the fundamental class of $X_i \subset X$ is $2g \cdot z_i$. It follows that $z_i = 0 \in A^*(X \smallsetminus (X_1 \cup \cdots \cup X_n))$. Since $0 = z_1 + z_2 + c_1$, we also see that $c_1 = 0 \in A^*(X \smallsetminus (X_1 \cup \cdots \cup X_n))$.
Finally, the form $F$ should not have any triple roots and must have at least one simple root, so we also wish to remove $U \times_{\BGL_2} (\Delta_3 \cup R) \subset U\times_{\BGL_2} \Sym^{2g+2} \V^\vee = X$. This introduces the relations from Lemma \ref{0pt}, which showed $c_2$ is a multiple of $c_1^2$.
In particular,
\[A^*(X \smallsetminus (X_1 \cup \cdots \cup X_n \cup U \times_{\BGL_2} (\Delta_3 \cup R))) = \qq.\] 
The above open subset of $X$ has the same coarse space as $W_n$, so $A^*(W_n) = \qq$ as well.
\end{proof}

We now describe many relations in $R^*(\I_{g,n})$, which almost determine this ring.
\begin{lem} \label{r2}
All classes in $R^2(\I_{g,n})$ are proportional to $\delta^2$. In other words,
$\dim R^2(\I_{g,n}) \leq 1$.
Moreover, $R^i(\I_{g,n}) = 0$ for all $i \geq 3$. 
\end{lem}
\begin{proof}
By Lemma \ref{tcor}, we know $R^*(\I_{g,n})$ is generated by the $\psi$ classes and $\delta$. Moreover, by Lemma \ref{0pt}, we know $\delta^3 = 0$. Hence, it suffices to show $\psi_i^2$ and $\psi_i\psi_j$ are proportional to $\delta^2$.

The class $\psi_i$ is pulled back from the map $\I_{g,n} \to \I_{g,1}$ that forgets all but the $i^{\text{th}}$ marked point, so Lemma \ref{1pt} shows that $\psi_i^2$ and $\psi_i\delta$ are proportional to $\delta^2$.

For $i \neq j$, the intersection of $D_{ii}$ and $D_{ij}$ in $\I_{g,n}$ is empty because at any point of their intersection we would have $p_i = \overline{p}_i =p_j$, which is impossible. 
Using \eqref{dii} and \eqref{dij}, we see
\begin{align*}
0 &= [D_{ii}][D_{ij}]
= \frac{g+1}{(g-1)^2} \psi_i \psi_j + \text{terms proportional to $\delta^2$}.
\end{align*}
so $\psi_i\psi_j$ is also proportional to $\delta^2$.
\end{proof}

\begin{rem}
Lemma \ref{r2} nearly determines $R^*(\I_{g,n})$. All that remains unknown is if $\delta^2$ is non-zero. (We know $\delta^2$ is non-zero on $\I_{g,0}$ and $\I_{g,1}$, but it is possible that $\delta^2$ could lie in the kernel of the pullback to $\I_{g,n}$ for some $n > 1$.)
\end{rem}


Since $\delta$ lies in the kernel of the restriction map $R^*(\I_{g,n}) \to R^*(\H_{g,n})$ and $\psi_1, \ldots, \psi_n$ are independent in $A^1(\H_{g,n})$, Lemma \ref{r2} implies
 Proposition \ref{rprop} from the introduction:
\begin{cor}
We have $R^*(\H_{g,n}) = \qq[\psi_1, \ldots, \psi_n]/(\psi_1, \ldots, \psi_n)^2$.
\end{cor}


\section{Quotient stack presentations}\label{quotstack}
The purpose of this section is to give quotient stack presentations for the stacks $\I_{g,n}^{\circ, i}$ and $W_i$ (defined in \eqref{Hl}--\eqref{si}). Using these presentations, we will show that $A^*(\I_{g,n}^{\circ, i})$ is generated by divisors and $A^*(W_{i})$ is generated by restrictions of divisors from $\I_{g,n}$. 

We begin by constructing models for our hyperelliptic curve lying on $\pp^1 \times \pp^1$, in which we will be able to keep track of our marked points. Our method is inspired by Casnati's work \cite{Casnati}, which used plane models of hyperelliptic curves. One of the maps to $\pp^1$ is given by the hyperelliptic map $\alpha:C \to \pp^1$. We denote the corresponding degree two line bundle by $L := \alpha^*\O_{\pp^1}(1)$. The other map to $\pp^1$ has degree $g+1$ and comes from the following lemma.

Recall our convention \eqref{conv} that $\I_{g,n}^{\circ, i} = \I_{g,n}^\circ$ if $n \geq g+1$.

\begin{lem} \label{themap}
Given $(C, p_1, \ldots, p_n) \in \I_{g,n}^{i, \circ}$, define
\[M = \begin{cases}  \O(p_1 + \ldots + p_{i-1} + (g-n+2)p_{i} + p_{i+1} + \ldots + p_n) & \text{if $n \leq g$} \\
\O(p_1 + \ldots + p_{g+1}) & \text{if $n \geq g+1$} \end{cases} .\]
Then $M$ is base point free with $h^0(C, M) = 2$.
\end{lem}
\begin{proof}
To start, observe $h^0(C, M) \geq \chi(M) = 2$ by Riemann--Roch.
Suppose for the sake of contradiction that there exists a point $p \in C$ so that $h^0(C, M(-p)) \geq 2$. Once we derive a contradiction, this will show $M$ is base point free and has exactly $2$ sections.

Let $M' = M(-p)$. Then $M'$ has degree $g$ and $h^0(C, M') \geq 2$. By Lemma \ref{georr}, we see $M' \cong L(x_1 + \ldots + x_{g-2})$, so $M \cong L(x_1 + \ldots + x_{g-2} + p)$.
It follows that 
\[\O(p_2 + \ldots + p_{i-1} + (g-n+2)p_i + p_{i+1} + \ldots + p_n) \cong  M(-p_1) \cong \O(\overline{p}_1 + x_1 + \ldots + x_{g-2} + p).\]
If $i \neq 1$, the assumption $(C, p_1, \ldots, p_n) \in \I_{g,n}^{\circ, i}$ means $\overline{p}_1$ is not among the points 
\[p_2 + \ldots + p_{i-1} + (g-n+2)p_i + p_{i+1} + \ldots + p_n.\] 
Thus, $h^0(C, M(-p_1)) \geq 2$. By Lemma \ref{georr}, we would have $p_j = p_k$ for some $j \neq k$, but this contradicts the fact that $(C, p_1, \ldots, p_n) \in \I_{g,n}^{\circ, i}$. 
If $i = 1$, then $\overline{p}_1 \neq p_1$ so $\overline{p}_1$ is still not among the points representing $M(-p_1) = (g-n+1)p_1 + p_2 + \ldots + p_n$. Thus, as before, $h^0(C, M(-p_1)) \geq 2$, so Lemma \ref{georr} gives a contradiction with the fact that no pair of points in this divisor is conjugate.
\end{proof}

Let $\beta: C \to \pp^1$ be the degree $g+1$ map associated to the line bundle $M$ introduced in Lemma \ref{themap}, so we have maps
\begin{equation} \label{ab}
\begin{tikzcd}
C \arrow{r}{\beta} \arrow{d}[swap]{\alpha} & \pp^1 = \pp H^0(C, M)^\vee \\
{\color{white} \pp H^0(C, L)^\vee = } \pp^1 = \pp H^0(C, L)^\vee 
\end{tikzcd}
\end{equation}
 The product $\alpha \times \beta: C \to \pp^1 \times \pp^1 $ sends $C$ to a curve of bidegree $(g+1, 2)$ on $\pp^1 \times \pp^1$.
By the degree-genus formula on $\pp^1 \times \pp^1$, the image has genus $g$, so the map must be an embedding. Let $\O(g+1,2) = \O_{\pp^1}(g+1) \boxtimes \O_{\pp^1}(2)$
The equation of $C \subset \pp^1 \times \pp^1$ has the form
\begin{align} \label{PQR}
    P(x_0, x_1)y_0^2 + Q(x_0, x_1)y_0y_1 + R(x_0, x_1)y_1^2 &\in H^0(\pp^1 \times \pp^1, \O(g+1,2)) \\
    &= H^0(\p^1,\O_{\pp^1}(g+1))\otimes H^0(\p^1,\O_{\pp^1}(2)) \notag
\end{align}
where $P, Q, R$ are homogeneous of degree $g+1$.
We call this vector space of equations 
$E:=H^0(\pp^1 \times \pp^1, \O(g+1,2))$.

We will need to know when certain configurations of points on $\pp^1 \times \pp^1$ impose independent conditions on forms in $E$.
\begin{lem}\label{indconditions}
Let $C\subset \p^1\times\p^1$ be an irreducible bidegree $(g+1,2)$ curve. Let $\Gamma$ be a degree at most $2g+5$ divisor contained in the smooth locus of $C$. Then the evaluation map
\begin{equation} \label{evgamma}
E=H^0(\p^1\times \p^1,\O(g+1,2))\rightarrow H^0(\Gamma,\O(g+1,2)|_{\Gamma})
\end{equation}
is surjective.
\end{lem}
\begin{proof}
The evaluation map \eqref{evgamma}
factors as
\begin{equation} \label{factorit}E = H^0(\pp^1 \times \pp^1, \O(g+1,2)) \to H^0(C, \O(g+1,2)|_C) \to H^0(\Gamma, \O(g+1,2)|_{\Gamma}).
\end{equation}
Write $N:=\O(g+1,2)|_C = L^{\otimes g+1} \otimes M^{\otimes 2}$, which is a line bundle of degree $4g+4$. By Riemann--Roch, we have $h^0(C, N) = \chi(C,N) = 3g+5$. Meanwhile,
\[h^0(\pp^1 \times \pp^1, \O(g+1,2)) = (g+2)(3) = 3g+6.\]
There is a unique polynomial of bidegree $(g+1,2)$ vanishing on $C$, so the first map in 
\eqref{factorit} is surjective by dimension counting.

Next, consider the exact sequence of sheaves on $C$:
\[ 0 \rightarrow N(-\Gamma) \to N \to N|_{\Gamma} \rightarrow 0\]
Taking global sections, we see that the second map in \eqref{factorit} is surjective if \[
0 = H^1(C, N(-\Gamma)) = H^0(\omega_C\otimes N(-\Gamma)^\vee).\] 
Because $n \leq 2g+5$, we see that 
\[\deg(\omega_C \otimes N(-\Gamma)^\vee) = 2g - 2 - (4g+4 - n) < 0,\]
so indeed it has no global sections.
\end{proof}


\subsection{The stack of appropriately marked $(g+1, 2)$ curves} \label{Xsec}
The construction modeling $C$ on $\pp^1 \times \pp^1$ 
works in families. Given $f: C \to S$ a family of irreducible nodal hyperelliptic curves with sections $p_1, \ldots, p_n: S \to C$, we
obtain a relative degree $g+1$ line bundle $M$ as in Lemma \ref{themap} (where, in the definition of $M$, we interpret $p_i$ as the image of the $i^{\mathrm{th}}$ section).
Since we have a marked point, there is a relative degree $2$ line bundle inducing the hyperelliptic map defined on the entire family by $L = \O(p_1 + \overline{p}_1)$.
Let us write $V = (f_*L)^\vee$ and $W = (f_*M)^\vee$, which are both rank $2$ vector bundles on $S$.
The relative version of \eqref{ab} becomes
\begin{equation} \label{abrel}
\begin{tikzcd}
C \arrow{r}{\beta} \arrow{d}[swap]{\alpha} &  \pp W = \pp (f_*M)^\vee\\
{\color{white} \pp(f_* L)^\vee = } \pp V = \pp (f_*L)^\vee
\end{tikzcd}
\end{equation}
We saw earlier that this map is an embedding on each fiber over $S$, so we obtain an embedding of our family $\iota: C \to \pp V \times_S \pp W$.
The sections $p_1,\ldots, p_n:S \to C$ give rise to sections $\sigma_1, \ldots, \sigma_n: S \to \pp V \times_S \pp W$ defined via $\sigma_{i} := \iota \circ p_i$.

In this way, we see $\I_{g,n}^{\circ, i}$ is equivalent to a stack $\X_{g,n, i}$ of appropriately marked $(g+1, 2)$ curves. More precisely, the objects of $\X_{g,n,i}$ over a scheme $S$ are tuples
\[(V, W, C, \sigma_1, \ldots, \sigma_n)\]
where $V, W$ are rank $2$ vector bundles on $S$; $C \subset \pp V \times_S \pp W$ is the zero locus of a  section of $\O_{\pp V}(g+1) \boxtimes \O_{\pp W}(2)$
such that $C \to S$ is a family of irreducible, at worst nodal curves; and $\sigma_1, \ldots, \sigma_n: S \to \pp V \times_S \pp W$ are sections that satisfy the following conditions:
\begin{enumerate}
    \item[(X.1)] If $\pi_1: \pp V \times_S \pp W \to \pp V$ is the first projection, then $\pi_1(\sigma_j) \cap \pi_1(\sigma_k) = \varnothing$ for $j \neq k$ (this ensures the sections do not give rise to conjugate points on $C$)
    \item[(X.2)] If $\pi_2: \pp V \times_S \pp W \to \pp W$ is the second projection, then {\small
    \begin{equation} \label{tcon} \pi_2^{-1}(\pi_2(\sigma_1)) \cap C = \begin{cases} \sigma_1 + \ldots + \sigma_{i-1} + (g-n+2) \sigma_i + \sigma_{i+1} + \ldots + \sigma_n & \text{if $n \leq g$} \\ 
\sigma_1 + \ldots + \sigma_{g+1} & \text{if $n \geq g+1$.}
\end{cases} \end{equation}}
 \item[(X.3)] $\sigma_1, \ldots, \sigma_n$ factor through the smooth locus of $C \subset \pp V\times_S \pp W \to S$
\end{enumerate}

Note that if $n \leq g$, then the fact that $\sigma_i$ is in the smooth locus of $C \to S$ and is tangent to the horizontal fiber $\pi_2^{-1}(\pi_2(\sigma_1))$ means $\sigma_i$ is not tangent to the vertical fiber $\pi_1^{-1}(\pi_2(\sigma_i))$, so it does not define a Weierstrass point on $C$.
See Figure \ref{pts-fig} below for an illustration of conditions (X.1) and (X.2).

\begin{figure}[h!]
    \centering
     \includegraphics[width=3.1in]{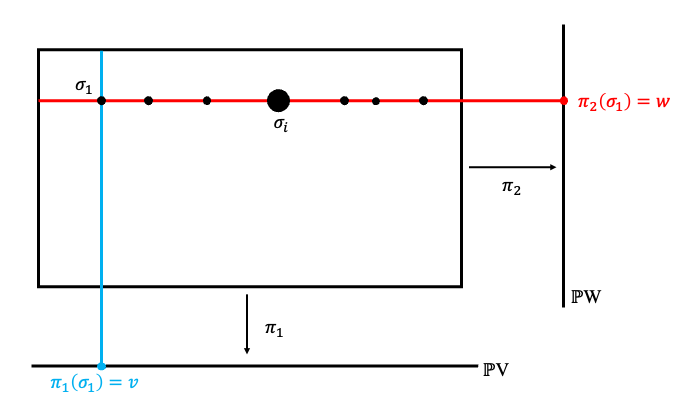}
    \hspace{.2in}
    \includegraphics[width=3.1in]{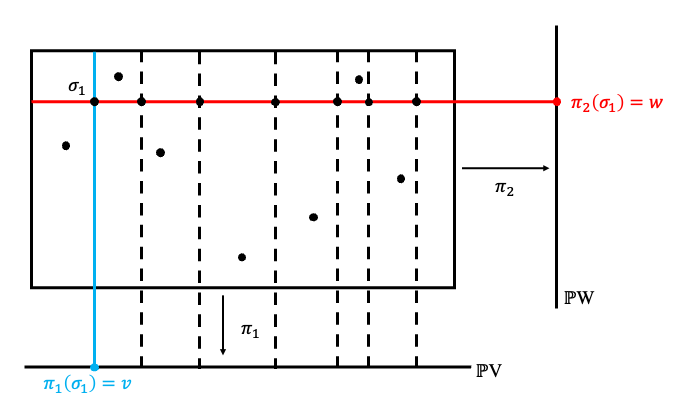}
   
    \caption{On the left is the configuration of points for $n \leq g$. In condition (X.2) We are going to ask that $C$ meet the red line with multiplicity $g-n+2$ at $\sigma_i$. On the right is the configuration of points for $n \geq g+1$. In both cases, the points have distinct projections onto the horizontal line (condition (X.1)). 
    }
    \label{pts-fig}
\end{figure}

\subsubsection{A savings of one point} \label{savings}
If $n \geq g+1$, then it is actually not necessary to keep track of $\sigma_{g+1}$. 
We can replace the condition \eqref{tcon} with the condition that $\pi_2^{-1}(\pi_2(\sigma_1))$ is reduced and contains $\sigma_1 + \ldots + \sigma_{g}$.
Indeed, from $C$ and $\sigma_1, \ldots, \sigma_g$, we can then recover $\sigma_{g+1}$ as $\pi_2^{-1}(\pi_2(\sigma_1)) \smallsetminus (\sigma_1 + \ldots + \sigma_g)$.

\subsection{Stacks of markings on $\pp^1 \times \pp^1$}
There is of course a natural map from $\X_{g,n,i}$ to the stack $\B_{g,n}$ whose objects over $S$ are tuples
\[(V, W, \sigma_1, \ldots, \sigma_g, \sigma_{g+2}, \ldots, \sigma_n)\]
where $V, W$ are rank $2$ vector bundles on $S$ and
$\sigma_1, \ldots, \sigma_g, \sigma_{g+2}, \ldots, \sigma_n: S \to \pp V\times_S \pp W$ are sections satisfying
\begin{enumerate}
    \item $\pi_1(\sigma_j) \cap \pi_1(\sigma_k) = \varnothing$ for $j \neq k$
    \item $\pi_2(\sigma_j) = \pi_2(\sigma_1)$ for $j \leq g$ and $\pi_2(\sigma_j) \cap \pi_2(\sigma_1) = \varnothing$ for $j \geq g+2$.
\end{enumerate}
(See Section \ref{savings} for why we are leaving out $\sigma_{g+1}$ when $n \geq g+1$.)

Finally, we consider the stack whose objects are tuples $(S, V, W, \sigma_1)$ with $V, W$ rank $2$ vector bundles on $S$ and $\sigma_1: S \to \pp V \times_S \pp W$ a section. The section $\sigma_1$ is equivalent to the data of sections $v: S \to \pp V$  and $w: S \to \pp W$. Let $\PU \subset \PGL_2$ be the stabilizer of a point $[1:0] \in \pp^1$. In other words, $\PU$ is the projectivization of the group $\mathrm{U} \subset \GL_2$ of upper triangular matrices.
The stack parametrizing the data of a rank $2$ bundle $V$ together with a section of $\pp V$ is the classifying stack $\BPU$. Thus, the stack of tuples $(S, V, W, \sigma_1)$ is simply $\BPU \times \BPU$.
The following diagram summarizes the maps we have considered so far
\begin{center}
\begin{tikzcd}
\I_{g,n}^{\circ, i} \cong \X_{g,n,i} {\color{white} \I_{g,n}^{\circ, i}} \arrow{d} & (C, p_1, \ldots, p_n) \arrow{r} &  \arrow{l} (V, W, C, \sigma_1, \ldots, \sigma_n) \arrow{d}  \\
\B_{g,n} \arrow{d}  & & (V, W, \sigma_1, \ldots, \sigma_g, \sigma_{g+2}, \ldots, \sigma_n) \arrow{d}   \\
\BPU \times \BPU & & (V, W, \sigma_1).
\end{tikzcd}
\end{center}
Our goal is now to factor each vertical map above into
open inclusions and affine/projective bundle morphisms.

First, we describe the map $\B_{g,n} \to \BPU \times \BPU$.
Let us write $\V$ and $\W$ for the universal rank $2$ bundles over $\BPU \times \BPU$, one pulled back from each factor. Let $v \times w: \BPU \times \BPU \to \pp \V \times \pp \W$ be the universal section and label the projection maps as follows
\begin{center}
\begin{tikzcd}
\pp \V \times \pp \W \arrow{r}{\pi_2} \arrow{d}[swap]{\pi_1} \arrow{dr}{\pi} & \pp \W \arrow{d} \\
\pp \V \arrow{r} & \BPU \times \BPU \ar[bend right = 30, u, swap, "w"]
\ar[bend left = 20, l, "v"]
\end{tikzcd}
\end{center}

In the case $n = 1$, we have $\B_{g,1} = \BPU \times \BPU$ and $\sigma_1 = v \times w$.
For $2 \leq n \leq g$, we have that $\B_{g,n}$ is an open substack of
\[\mathcal{A}_{g,n} := (\pp \V \smallsetminus v)^{n-1}. \]
This is saying that 
the sections $\sigma_2, \ldots, \sigma_n$ tell us $n-1$ points distinct from $\sigma_1$ on the same horizontal fiber as $\sigma_1$.
Meanwhile, if $n \geq g+1$, then $\B_{g,n}$ is an open substack of
\[\mathcal{A}_{g,n} := (\pp \V \smallsetminus v)^{g-1} \times ((\pp \V \smallsetminus v) \times 
(\pp \W \smallsetminus w))
^{n-g-1}.\]
This is because we have $g-1$ sections $\sigma_2, \ldots, \sigma_{g}$ on the same horizontal fiber as $\sigma_1$ and the remaining sections $\sigma_{g+2},\ldots, \sigma_n$ lie in the complement of the vertical and horizontal fibers through $\sigma_1$. The condition (X.1) that no two sections lie in the same vertical fiber is open.

\subsection{Relating $\X_{g,n,i}$ to a projective bundle over $\B_{g,n}$}
Next, let 
\[\N := \O_{\pp \V}(g+1) \boxtimes \O_{\pp \W}(2).\]
(This is the line bundle that defines curves in the class of interest on $\pp^1 \times \pp^1$.) Write $\pi: \pp \V \times \pp \W \to \BPU \times \BPU$ for the structure map. Define 
\[\E := \pi_* \N = \Sym^{g+1} \V^\vee \otimes \Sym^2 \W^\vee,\]
which is a vector bundle on $\BPU \times \BPU$.
Equivalently, $\E$ is the quotient $E/(\PU \times \PU)$ of the vector space $E$ we considered earlier in \eqref{PQR}.

There is a natural map from $\X_{g,n,i}$ to $\pp(\phi^*\E)$ over $\B_{g,n}$ that sends $(V, W, C, \sigma_1, \ldots, \sigma_n)$ to the equation defining $C \subset \pp \V \times \pp \W$. The image is contained in the locus of equations
with appropriate vanishing behavior at the sections $\sigma_j$.
Let $\phi: \B_{g,n} \to \BPU \times \BPU$  so we have
\begin{equation} \label{setup}
\begin{tikzcd}
& \pp \V \times \pp \W \arrow{d}{\pi} \\
 \arrow{ur}{\sigma_j} \B_{g,n} \arrow{r}[swap]{\phi} & \BPU \times \BPU.
\end{tikzcd}
\end{equation}
The equations in $\phi^* \E$ that vanish along $\sigma_j$ is just the kernel of the evaluation map pulled back from the $j^{\mathrm{th}}$ factor:
\[\phi^*\E = \eta_j^* \pi^* \pi_*\N \to \eta_j^* \N. \]

\subsection{The case $n \leq g$} \label{smn}
When $n \leq g$, we also want the vanishing of our equation to be tangent with order $g-n+2$ to the horizontal fiber through $\sigma_i$.
Such equations make up the kernel of the evaluation map in the principal parts bundle pulled back from the $i^{\mathrm{th}}$ factor:
\[\phi^*\E = \eta_i^* \pi^* \pi_*\N \to \eta_i^* P^{g-n+1}_{\pp \V \times \pp \W/\pp \W}(\N).\]
To get the correct vanishing behavior at all of $\sigma_1, \ldots, \sigma_n$, we are therefore interested in the kernel of the evaluation map
\begin{equation} \label{eeev}\mathrm{ev}: \phi^* \E \to \eta_i^*P^{g-n+1}_{\pp \V \times \pp \W/\pp \W}(\N) \oplus \bigoplus_{\substack{1 \leq j \leq n \\ j \neq i}} \eta_j^* \N.
\end{equation}
\begin{lem} \label{evs}
The map $\mathrm{ev}$ in \eqref{eeev} is surjective.
\end{lem}
\begin{proof}
The equation \eqref{eeev} is the equivariant version of the map 
\[E = H^0(\pp^1 \times \pp^1, \O(g+1, 2)) \to H^0(\O_{\pp^1}(g+1)) \to \O_D\]
that restricts a bidegree $(g+1, 2)$ polynomial to the distinguished horizontal line of the ruling $\pp \V \times \pi_2(\sigma_1)$ and then restricts the resulting degree $g+1$ polynomial to the degree $g+1$ subscheme $D = \sigma_1 + \ldots + (g+n-2)\sigma_i + \ldots+\sigma_n \subset \pp \V \times \pi_2(\sigma_1)$.
 If we were to pick coordinates as in \eqref{PQR}, the first map would be taking the polynomial $P(x_0, x_1)$, and the second map would be evaluating $P(x_0, x_1)$ at a specified degree $g+1$ subscheme of $\pp^1$.
Both of these maps are surjective.
\end{proof}

Let $\F \subset \phi^*\E$ be the kernel of $\mathrm{ev}$, which is a vector bundle by Lemma \ref{evs}.

\begin{lem} \label{oplem}
There is an open inclusion $\I_{g,n}^{\circ}\cong\X_{g,n,i} \subset \pp \F$.
\end{lem}
\begin{proof}
The bundle $\pp\F$ over $\B_{g,n}$ parametrizes $(V, W, C, \sigma_1, \ldots, \sigma_n)$ satisfying (X.1) and (X.2) in Section \ref{Xsec}. The additional conditions that define $\X_{g,n,i}$ are that $C$ is a family of irreducible nodal curves and (X.3), both of which are open conditions.
\end{proof}

\subsubsection{The strata $W_{i-1}$ when $n \leq g$}
Using a similar ideal, we can realize $W_{i-1} \subset \I_{g,n}^{\circ, i}$ as the intersection of $\I_{g,n}^{\circ, i}$ with a subbundle $\pp \F_{i-1} \subset \pp \F$ for an appropriate vector subbundle $\F_{i-1} \subset \F$. The $j^{\mathrm{th}}$ marked point will be Weierstrass if the equation for $C$ is tangent to the vertical fiber through $\sigma_j$. This means we need to consider the evaluation map
\begin{equation} \label{evi}\mathrm{ev}_{i-1}: \phi^* \E \to \bigoplus_{1 \leq j \leq i-1} \eta_j^* P^1_{\pp \V \times \pp \W/\pp \V}(\N) \oplus \eta_i^*P^{g-n+1}_{\pp \V \times \pp \W/\pp \W}(\N) \oplus \bigoplus_{i+1\leq j \leq n} \eta_j^*\N. \end{equation}

\begin{figure}[h!]
    \centering
     \includegraphics[width=4in]{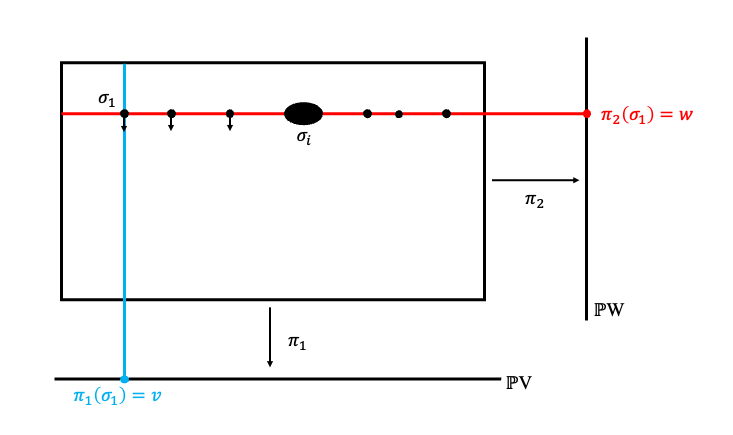}
    \caption{The points to the left of $\sigma_i$ are Weierstrass points, indicated by the vertical tangent arrows to each point.}
    \label{weierstrassfig}
\end{figure}

\begin{lem}
The evaluation map $\mathrm{ev}_{i-1}$ in \eqref{evi} is surjective.
\end{lem}
\begin{proof}
Let us describe the map in terms of the vector space $E$ with coordinates as in \eqref{PQR}. As in Lemma \ref{evs}, evaluation in $\eta_j^*\N$ is the equivariant version of evaluating $P(x_0, x_1)$ at various points along $\pp^1$. The evaluation in $\eta_i^*P^{e}_{\pp \V \times \pp \W/\pp \W}(\N)$ is evaluating $P(x_0, x_1)$ in an $e^{\mathrm{th}}$ order neighborhood around the $i^{\mathrm{th}}$ point in $\pp^1$. 

On the other hand, evaluation in the rank $2$ bundle $\eta_j^*P^1_{\pp \V \times \pp \W/\pp \V}(\N)$ corresponds to evaluating $P(x_0, x_1)$ and $Q(x_0, x_1)$ at the $j^{\mathrm{th}}$ point on $\pp^1$. Since $Q(x_0, x_1)$ has degree $g+1$, the evaluation map at $i - 1 \leq g$ points is surjective.
\end{proof}

There are natural surjections $\eta_j^* P^1_{\pp \V \times \pp \W/\pp \V}(\N)  \to \eta_j^*\N$, so
the target of \eqref{eeev} is a quotient of the target of \eqref{evi}. Hence,
$\F_{i-1} := \ker(\mathrm{ev}_{i-1})$ is a subbundle of $\F = \ker(\mathrm{ev})$.
Moreover, $W_{i-1} \subset \I_{g,n}^{\circ, i}$ is the intersection of $\I_{g,n}^{\circ, i}$ with $\pp \F_{i-1} \subset \pp \F$, so we obtain a diagram where the square at the top is fibered:
\begin{equation}
\begin{tikzcd}
W_{i-1} \arrow{r} \arrow{d} &\pp \F_{i-1} \arrow{d} \\
 \I_{g,n}^{\circ,i} \arrow{r} & \pp \mathcal{F} \arrow{d} \\
  & \B_{g,n} \arrow{r} & \mathcal{A}_{g,n} \arrow{d} \\
&&  \BPU \times \BPU
\end{tikzcd}
\end{equation}
Above, the vertical maps are projective/affine bundles and the horizontal maps are open inclusions.

\begin{lem} \label{dglem}
We have $A^*(\I_{g,n}^{\circ, i})$ is generated by divisors, and
$A^*(W_{i-1})$ is generated by restrictions of divisors from $A^*(\I_{g,n})$.
\end{lem}
\begin{proof}
By excision, to prove the first claim, it suffices to show $A^*(\pp \F)$ is generated by divisors. By the projective bundle theorem, this follows if $A^*(\B_{g,n})$ is generated by divisors. By excision again, this follows if $A^*(\mathcal{A}_{g,n})$ is generated by divisors. But $\mathcal{A}_{g,n}$ is an affine bundle over $\BPU \times \BPU$, so it suffices to show that $A^*(\BPU \times \BPU)$ is generated by divisors, which we see as follows.

The map $\PU\rightarrow \gg_m \ltimes \gg_a$
\[
A=\begin{pmatrix}
a & b \\
0 & c
\end{pmatrix}\mapsto (a/c,b/c)
\]
is an isomorphism. This induces a natural map 
\[
\mathrm{B}\gg_m\times \mathrm{B}\gg_m\rightarrow \BPU\times \BPU,
\]
which is an affine bundle with fiber $\mathbb{A}^2$, so the pullback map induces an isomorphism on Chow rings. The Chow ring $A^*(\mathrm{B}\gg_m\times \mathrm{B}\gg_m)$ is generated by divisors, corresponding to the first Chern classes of the two universal line bundles. 

To prove the second claim, note that $\O_{\pp \F}(1)$ restricts to $\O_{\pp \F_{i-1}}(1)$. Thus, by the projective bundle theorem, $A^*(\pp \F) \to A^*(\pp \F_{i-1})$ is surjective. By excision $A^*(\pp \F_{i-1}) \to A^*(W_{i-1})$ is surjective. Hence, $A^*(\pp \F) \to A^*(W_{i-1})$ is surjective. But since the top square commutes, all classes in the image of $A^*(\pp \F) \to A^*(W_{i-1})$ are in the image of $A^*(\I_{g,n}^{\circ,i}) \to A^*(W_{i-1})$. Finally, $A^*(\I_{g,n}) \to A^*(\I_{g,n}^{\circ, i})$ is surjective by excision, so $A^*(\I_{g,n}) \to  A^*(\I_{g,n}^{\circ, i}) \to A^*(W_{i-1})$ is surjective.
\end{proof}

We now conclude the proof of Theorem \ref{divgen} in the case $n \leq g$
\begin{proof}[Proof of Theorem \ref{divgen} over $\cc$ when $n \leq g$]
By Theorem \ref{Scavia}, it suffices to show that $A^*(\I_{g,n})$ is generated by divisors.
By Lemma \ref{dglem}, we have that $W_i$ is generated by restrictions of divisors from $\I_{g,n}$ for $i \leq n-1$. (Note $W_0 = \I_{g,n}^{\circ, 1}$.) Lemma \ref{Wn} shows $W_n$ is also generated by restrictions of divisors from $\I_{g,n}$. The fundamental class of each $W_i$ is a product of divisors (Lemma \ref{fc}), so we conclude that $A^*(\I_{g,n}^\circ) = A^*(W_0 \cup W_1 \cup \cdots \cup W_n)$ is generated by divisors for $n \leq g$.

Now we proceed by induction. The base case is that $A^*(\I_{g,0})$ is generated by divisors (Lemma \ref{0pt}).
Suppose we know that $A^*(\I_{g,n-1})$ is generated by divisors. By Proposition \ref{induction}, we know $A^*(D_{ij})$ is generated by restrictions of divisors from $\I_{g,n}$ for each $i \neq j$. By the push-pull formula, all classes supported on $D_{ij}$ are products of divisors. Since $A^*(\I_{g,n}^\circ)$ is generated by divisors, excision shows that $A^*(\I_{g,n})$ is generated by divisors.
\end{proof}

\subsection{The case $g+1 \leq n \leq 2g+6$} \label{ng1}
When $n \geq g+1$, we have no higher order vanishing conditions to worry about at $\sigma_i$. Recall also that $\I_{g,n}^{\circ, i} = \I_{g,n}^{\circ}$. To find the equations in $\phi^*\E$ that vanish along $\sigma_1, \ldots, \sigma_g,\sigma_{g+2},\ldots, \sigma_n$, we consider the evaluation map
\begin{equation} \label{ev2}
\mathrm{ev}: \phi^*\E \to \bigoplus_{\substack{1 \leq j \leq n \\ j \neq g+1}} \eta_j^*\N.
\end{equation}
The preimage of the zero section of \eqref{ev2} inside $\phi^* \E$ parametrizes global sections of $\N$ that vanish along $\sigma_1, \ldots, \sigma_g, \sigma_{g+2},\ldots, \sigma_n$. By construction then, there is a map $\I_{g,n}^{\circ} \to \pp (\mathrm{ev}^{-1}(0)) \subset \pp (\phi^*\E)$. This map is an open inclusion by a similar argument to Lemma \ref{oplem}.

The problem now is that $\mathrm{ev}^{-1}(0)$ is not necessarily a vector bundle because the rank of \eqref{eeev} can drop.
Let $\B^\circ \subset \B_{g,n}$ be the locus where $\eqref{eeev}$ is surjective. 
By semicontinuity, we know $\B^\circ \subset \B_{g,n}$ is open. (Of course, $\B^\circ$ may be empty, and it will be when $n$ is too large.)
Define
$\F \subset \phi^*\E|_{\B^\circ}$ to be the kernel of the restriction of \eqref{ev2} to $\B^\circ$.
By definition of $\B^\circ$, we know that $\F$ is a vector bundle over $\B^\circ$. There is a fibered square as below, and we wish to know if $\I_{g,n}^{\circ, i} \to \pp(\mathrm{ev}^{-1}(0)) \subset \pp (\phi^*\E)$ factors through $\pp \F$.
\begin{center}
\begin{tikzcd}
\I_{g,n}^{\circ} \ar[drr, bend left = 10] \ar[ddr, bend right = 10, swap, dashed, "?"] \arrow[dashed, "?"]{dr} \\
& \pp \F \arrow{r} \arrow{d} & \pp (\mathrm{ev}^{-1}(0)) \arrow{d} \\
& \B^\circ \arrow{r} & \B_{g,n}
\end{tikzcd}
\end{center}
Equivalently, we wish to know if $\I_{g,n}^{\circ} \to \B_{g,n}$ factors through $\B^\circ$.
The following lemma shows that this holds when $n$ is sufficiently small relative to $g$.

\begin{lem} \label{2g6}
If $n \leq 2g + 6$, then the image of $\I_{g,n}^{\circ} \to \B$ is contained in $\B^\circ$.
\end{lem}
\begin{proof}
Each point in the image of $\I_{g,n}^{\circ} \to \B_{g,n}$ corresponds to a collection of $n-1 \leq 2g+5$ points on $\pp^1 \times \pp^1$ that lie in the smooth locus of an irreducible $(g+1, 2)$ curve. (We have $n-1$ points because we are not tracking $\sigma_{g+1}$, see Section \ref{savings}.)
Thus, surjectivity follows from Lemma \ref{indconditions}.
\end{proof}

\begin{rem}[The case $n \geq 2g+7$]
When $n \geq 2g+7$, Lemma \ref{2g6} fails. The reason is because $n - 1 \geq 2g+6$ points on an irreducible $(g+1,2)$ curve can fail to impose independent conditions. 
Looking back at Lemma \ref{indconditions}, 
we see this occurs  when the line bundle $\omega_C \otimes N(-\Gamma)^\vee$ has non-trivial global sections. The first case of this is when $\omega_C \otimes N(-\Gamma) \cong \O$ which is the case $\Gamma$ consists of $2g+6$ points 
which are the intersection of the $(g+1,2)$ curve $C$ with a $(2, 2)$ curve.
\end{rem}

It follows from Lemma \ref{2g6}  that $\I_{g,n}^{\circ}$ is an open substack of $\pp \F$.
Thus, we have a diagram
\begin{equation}
\begin{tikzcd}
 \I_{g,n}^{\circ} \arrow{r} & \pp \mathcal{F} \arrow{d} \\
 &\B^{\circ} \arrow{r} & \B_{g,n} \arrow{r} & \mathcal{A}_{g,n} \arrow{d} \\
&&&  \BPU \times \BPU
\end{tikzcd}
\end{equation}
where the vertical maps are projective or affine bundles and the horizontal maps are open inclusions.
Essentially the same argument as in Lemma \ref{dglem} yields the following.
\begin{lem} \label{key}
When $n \leq 2g+6$, we have $A^*(\I_{g,n}^{\circ})$ is generated by divisors.
\end{lem}

\begin{proof}[Proof of Theorem \ref{divgen} over $\cc$ when $g+1 \leq n \leq 2g+6$]
Again,
we argue by induction. Suppose that we know $A^*(\I_{g,n-1})$ is generated by divisors. Then the push-pull formula together with Proposition \ref{induction} shows that classes supported on $D_{ij}$ are products of divisors. By Lemma \ref{key}, we know $A^*(\I_{g,n}^{\circ})$ is generated by divisors, so the result follows by excision.
\end{proof}

\subsection{Rationality when $n \leq 3g+6$} A key step in the previous section was knowing that the image of $\I_{g,n}^\circ$ was contained in the locus $\B^\circ \subset \B_{g,n}$ where the evaluation map was surjective.
In this section, we explain how knowing the weaker fact that $\B^\circ$ is non-empty shows that $\I_{g,n}^\circ$ is rational.

\begin{lem}
Let $\B^\circ \subset \B_{g,n}$ be the locus where the evaluation map \eqref{ev2} is surjective. If $n \leq 3g+6$, then $\B^\circ$ is non-empty.
\end{lem}
\begin{proof}
This is equivalent to showing that a general collection of $n-1 \leq 3g+5$ points on $\pp^1 \times \pp^1$ with $g$ on the same line of the ruling impose independent conditions on $\O(g+1, 2)$. (Note that this is much weaker than our previous situation where we wanted to know \emph{every} collection of $n-1$ points lying on an irreducible $(g+1,2)$ curve impose independent conditions on $\O(g+1,2)$.)
We know any $g$ distinct points $p_1, \ldots, p_g$ on a line of the ruling impose independent conditions. Now continue choosing points $p_i$ not in the base locus of the linear system of $(g+1,2)$ curves vanishing at $p_1, \ldots, p_{i-1}$. This is possible so long as the kernel of the evaluation map $H^0(\pp^1 \times \pp^1, \O(g+1,2)) \to \bigoplus_{j=1}^{i-1} \O(g+1,2)|_{p_j}$ has dimension at least $1$. This holds at each step since we have $n - 1 \leq 3g+5$ points and  $h^0(\pp^1 \times \pp^1, \O(g+1,2)) = 3g+6$.
\end{proof}

Now, let $\F \subset \phi^*\E|_{\B^\circ}$ be the kernel of \eqref{ev2} restricted to $\B^\circ$. We obtain a diagram where both squares are fibered
\begin{center}
\begin{tikzcd}
U \arrow{r} \arrow{d} &\I_{g,n}^{\circ} \arrow{d} \\
\pp \F \arrow{r}  \arrow{d} & \pp(\mathrm{ev}^{-1}(0)) \arrow{d} \\
 \B^\circ \arrow{r} & \B_{g,n}
\end{tikzcd}
\end{center}
Although $\I_{g,n}^\circ \to \pp(\mathrm{ev}^{-1}(0))$ does not factor through $\pp \F$, since $\B^\circ \subset \B_{g,n}$ is non-empty there is still a dense open $U \subset \I_{g,n}^\circ$ which does factor through $\pp \F$. Moreover, because a general collection of $n-1 \leq 3g+5$ points impose indpendent conditions on $\O(g+1,2)$, the inclusion $U \subset \pp \F$ is open. Finally, it remains to show that $U$ is rational.

\begin{lem}
If $n \leq 3g+6$, then $U$ defined above is rational. Hence, $\I_{g,n}$ is rational.
\end{lem}
\begin{proof}
Our $U$ is an open subset of a projective bundle $\pp \F$ over $\B^\circ$, so it suffices to show that $\B^\circ$ is rational.
Since Casnati has proved rationality for $n \leq 2g+8$, we can assume $n \geq 2g+9 > g+3$. This ensures that there are at least two marked points which are not in the same fiber as $\sigma_1$.
On the open subset of $\B^\circ$ where $\pi_2(\sigma_{g+2}) \cap \pi_2(\sigma_{g+3}) = \varnothing$, we can ``use up" the $\PU \times \PU$ action to fix \[\sigma_1 = [1:0] \times [1:0], \sigma_{g+2} = [1:1] \times [1:1], \quad\text{and}\quad \sigma_{g+3} = [0:1] \times [0:1].\]
The other markings are then parametrized by an open subset of $(\pp^1)^{g-1} \times (\pp^1 \times \pp^1)^{n-g-3}$, which is rational.
\end{proof}

\section{Upgrading to a proof in characteristic $\neq 2$} \label{pp}
In the course of the proof of Theorem \ref{divgen} and Proposition \ref{rprop} so far, we relied on the following two facts that Scavia has shown to hold over $\cc$:
\begin{itemize}
    \item[(a)] The classes $\psi_1, \ldots, \psi_n \in A^1(\H_{g,n})$ are independent (so that there are no more relations in Proposition \ref{rprop})
    \item[(b)] The classes $\psi_1, \ldots, \psi_n, \delta$ span $A^1(\I_{g,n})$ when $n \leq 2g+6$ (to show these are the codimension $1$ generators in Theorem \ref{divgen}). Equivalently, the classes $\psi_1, \ldots, \psi_n$ span $A^1(\H_{g,n})$ when $n \leq 2g + 6$.
\end{itemize}
With the set up we have developed, it is not too much more work to establish (a) and (b) in any characteristic $\neq 2$, thereby proving Theorem \ref{divgen} more generally.
This also provides a new proof of Scavia's result that the $\psi$ classes are a basis for $A^1(\H_{g,n})$ when $n \leq 2g+6$, which holds in positive characteristic.

\subsection{Independence of $\psi$ classes} \label{psisec}
Our proof of independence is geometric, using test curves on a partial compactification of $\H_{g,n}$. The method is in part inspired by work of Edidin--Hu \cite{EdidinHu}, which used test curves to determine the classes of $\overline{D}_{11} \in A^1(\overline{\H}_{g,1})$ and $\overline{D}_{12}
\in A^1(\overline{\H}_{g,2})$. However, we shall only need the simplest families of test curves introduced there (and their generalizations to more marked points).

\begin{rem}
Edidin--Hu work over $\mathbb{C}$ in order to make use of Scavia's result that gives a basis for $A^1(\overline{\H}_{g,n})$. However, the test curves we use make sense over any ground field of characteristic $\neq 2$. Independence of divisors can be proved so long as we have enough test curves. In Section \ref{spanning}, we will show that the $\psi$ classes span $A^1(\H_{g,n})$ by dimension counting and excision.
\end{rem}

The following lemma is standard, but we include a proof for the convenience of the reader.

\begin{lem} \label{psitest}
Suppose $X \to C$ is a family of pointed curves with smooth total space $X$ and sections $\sigma_i : C \to X$. Let $a: C \to \Mb_{g,n}$ be the induced map. Then $\deg(a^*\psi_i) = -[\sigma_i(C)]^2$.
\end{lem}
In this situation, we shall often abbreviate $[\sigma_i(C)]^2$ by $\sigma_i^2$.
\begin{proof}
We have 
\[
a^*\psi_i = -c_1(\sigma_i^* \T_{X/C}) = -(c_1(\sigma_i^*\T_C) - c_1(\T_C)) = -c_1(N_{\sigma_i(C)/X}) = -[\sigma_i(C)]^2. \qedhere
\]
\end{proof}

Let
$\widetilde{\H}_{g,n} \subset \overline{\mathcal{H}}_{g,n}$
denote the open subset of pointed hyperellipitc curves $(C, p_1, \ldots, p_n)$ so that $C$ is irreducible
or a union of an irreducible component of genus $g$ and a genus $0$ curve containing exactly two marked points. We write $\Delta_{ij} \subset \widetilde{\H}_{g,n}$ for the divisor where $p_i$ and $p_j$ are the two points on the genus $0$ component. Let $\delta_{ij} = [\Delta_{ij}] \in A^1(\widetilde{\H}_{g,n})$.

\begin{lem} \label{psi-indep}
The classes $\psi_1, \ldots, \psi_n$ and the boundary divisors $\delta_{ij}$ (for $i \neq j$) are independent in $A^1(\tilde{\H}_{g,n})$ for any $n$ and $g \geq 2$. In particular, $\psi_1, \ldots, \psi_n$ are independent in $A^1(\H_{g,n})$ for any $n$ and $g \geq 2$.
\end{lem}
\begin{proof}
In total, we have $n + {n \choose 2}$ divisors we wish to show are independent.
Let us now introduce $n + {n \choose 2}$ test curves in $\widetilde{\H}_{g,n}$ which we shall intersect with our divisors. Our test curves come in two families. Below, we compute the intersection number of each test curve with each divisor. Then we describe a change of basis which makes the intersection matrix upper triangular with nonzero diagonal entries, and so visibly full rank.

\subsubsection{The first family of test curves:} \label{ff} First, for each $i = 1, \ldots, n$, we let $T_i$ be the test curve where $p_i$ roams over the curve and $p_j$ for $j \neq i$ is a fixed non-Weierstrass point.
To build this family,
take $C \times C$ and sections $\sigma_i = \Delta$ (the diagonal) and $\sigma_j = C \times p_j$ when $j \neq i$. Then we blow up the points $(p_j, p_j)$ for $j \neq i$.
Write $\nu:X \to C \times C$ for the blow up, and $E_j$ for the exceptional divisors. Then $T_i$ is the family $X \to C$ with sections $\tilde{\sigma}_1, \ldots, \tilde{\sigma}_n$ which are the proper transforms of $\sigma_1, \ldots, \sigma_n$. 
\begin{center}
\includegraphics[width=6in]{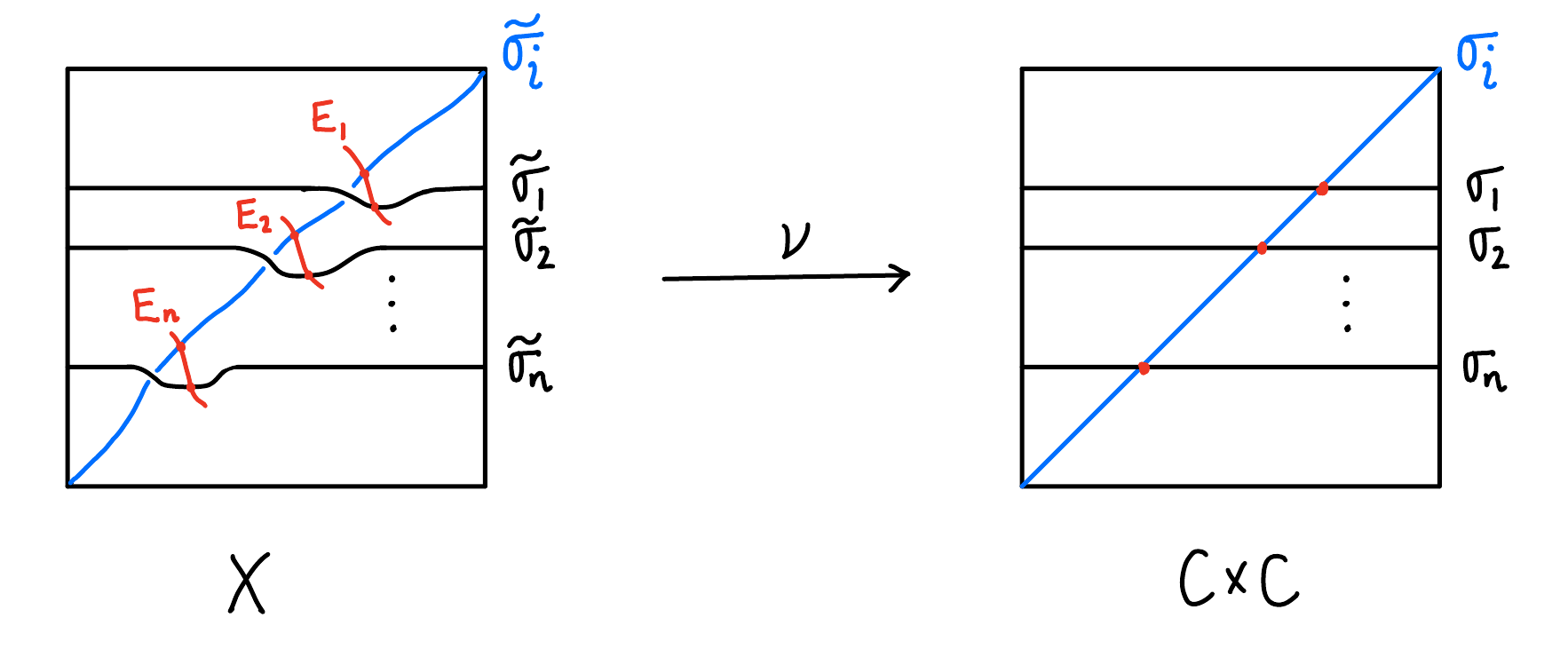}
\end{center}

Since $\sigma_i \subset C \times C$ is the diagonal, we have 
\[\sigma_i^2 = \deg N_{\sigma_i/X} = \deg T_{C \times C} - \deg T_C = -(2g-2).\]
Then we have
\[-(2g-2) = \sigma_i^2 = (\nu^*\sigma_i)^2 = (\tilde{\sigma}_i + \sum_{j \neq i} E_j)^2 = \tilde{\sigma}_i^2 + 2(n-1) - (n-1). \]
Hence, $\tilde{\sigma}_i^2 = -(2g + n - 3)$. In a similar manner, for $j \neq i$
\[0 = \sigma_j^2 = (\nu^*\sigma_j)^2 = (\tilde{\sigma}_j + E_j)^2 = \tilde{\sigma}_j^2 + 2 - 1,\]
so $\tilde{\sigma}_j^2 = -1$.
By Lemma \ref{psitest}, we therefore obtain
\begin{align} \label{i1}
T_i \cdot \psi_j &= \begin{cases} 1 & \text{if } i \neq j \\
2g + n - 3 & \text{if } i = j. \end{cases}
\intertext{Because the total space of our test family is smooth, the intersection of $T_i$ with each boundary divisor $\Delta_{jk}$ is reduced.
Therefore, we obtain}
T_i \cdot \delta_{jk} &= \begin{cases} 1 & \text{if $j = i$ or $k= i$} \\ 0 & \text{if $j, k \neq i$}.\end{cases}
\end{align}

\subsubsection{The second family of test curves}
Our second family of test curves are curves $T_{ij}$ where $p_i$ runs over the curve, $p_j$ is its conjugate, and the other points are fixed non-Weierstrass points.
To build this, we start with $C \times C$ and take sections $\sigma_i = \Delta$ and $\sigma_j = \iota(\Delta)$ where $\iota : C\times C  \to C\times C$ sends $(p, q) \mapsto (p, \overline{q})$.
Now $\sigma_i$ and $\sigma_j$ intersect at the points $(p, p)$ where $p$ is Weierstrass. For $k \neq i, j$ we set $\sigma_k = C \times p_k$ to be a fixed non-Weierstrass point. Then $\sigma_k$ meets $\sigma_i = \Delta$ in $(p_k, p_k)$ and it meets $\sigma_j = \iota(\Delta)$ in $(\overline{p}_k, p_k)$.
Let $\nu: S \to C \times C$ be the 
blow up at all points $(p, p)$ with $p$ Weierstrass (call these exceptionals $E_1', \ldots, E_{2g+2}'$) and all points $(p_k, p_k)$ and $(\overline{p}_k, p_k)$ for $k \neq i,j$ (call these exceptionals $E_k$ and $\overline{E}_k$ respectively.)
\begin{center}
\includegraphics[width=6in]{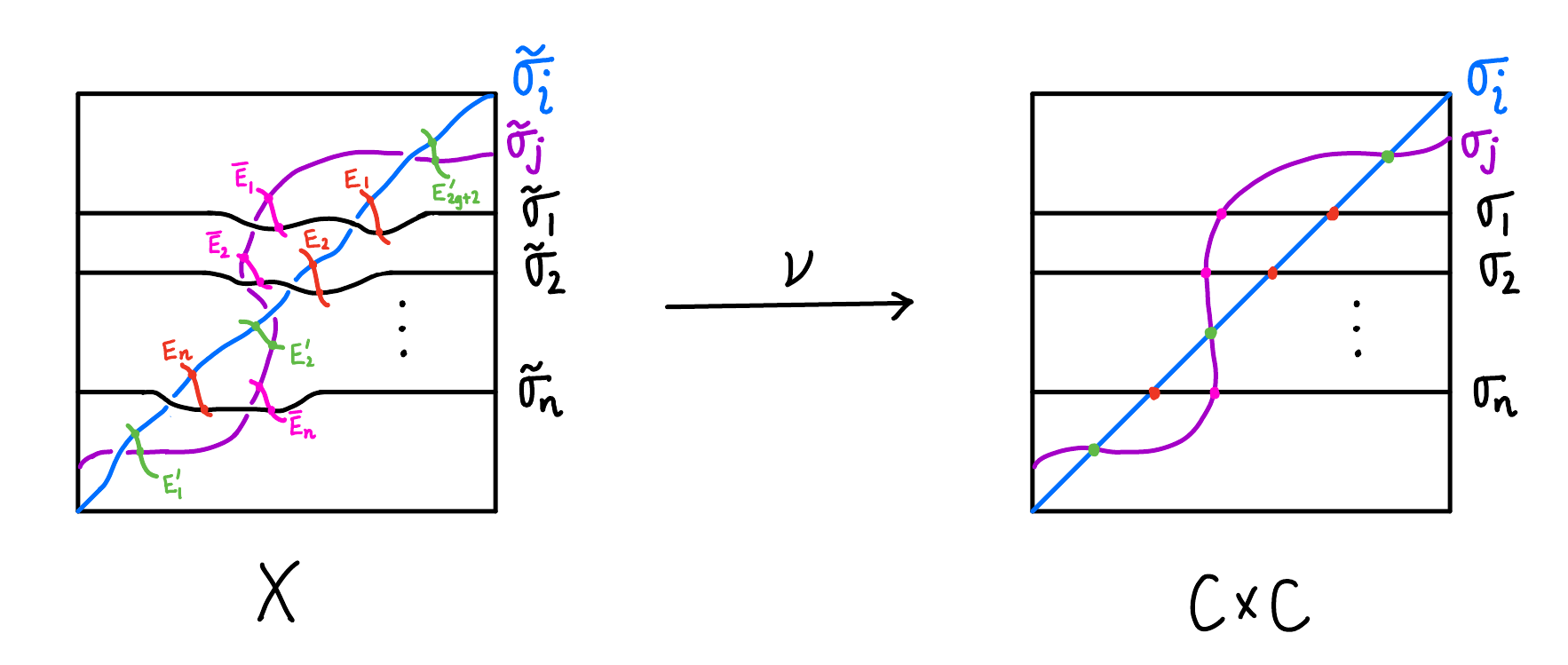}
\end{center}

To calculate the degree of the $\psi$ classes, we note that
\[-(2g-2) = (\nu^*\sigma_i)^2 = (\tilde{\sigma}_i + E_1'+\ldots + E_{2g+2}' + \sum_{k \neq i,j} E_k)^2 = \tilde{\sigma}_i^2 + 2(2g+2+n-2) - (2g+2+n-2) \]
and so $\tilde{\sigma}_i^2 = -(4g+n-2)$. A similar calculation (computing the self-intersection of the divisor $\nu^*\sigma_j = \tilde{\sigma}_j + E_1' + \ldots + E_{2g+2}' + \sum \overline{E}_k$) yields $\tilde{\sigma}_j^2 = -(4g+n-2)$ as well.
Meanwhile, for $k \neq i,j$, we have
\[0 = (\nu^*\sigma_k)^2 = (\tilde{\sigma}_k + E_k + \overline{E}_k)^2 =\tilde{\sigma}_k^2 + 4 - 2 \]
and so $\tilde{\sigma}_k = -2$.
Therefore, by Lemma \ref{psitest} we have
\begin{align}
T_{ij} \cdot \psi_k &= \begin{cases} 2 & \text{if $k \neq i, j$} \\
4g+n-2  & \text{if $k = i$ or $k = j$.} \end{cases}
\intertext{Again, the total space of our test family is smooth, so the intersection of $T_{ij}$ with the boundary divisor $\Delta_{k\ell}$ is reduced.
It is then straightforward to count}
T_{ij} \cdot \delta_{k\ell} &= \begin{cases}
2g+2 & \text{if }\#(\{i,j\} \cap \{k,\ell\}) = 2 \\
1 & \text{if }\#(\{i,j\} \cap \{k,\ell\})= 1 \\
0 & \text{if }\#(\{i,j\} \cap \{k,\ell\}) = 0. \label{i4}
\end{cases}
\end{align}

\subsubsection{The intersection matrix and change of basis}
Consider the intersection matrix with rows the test curves $T_i$ and $T_{ij}$ and the columns the divisors $\psi_i$ and $\delta_{ij}$.
For example, when $n = 3$, the intersection matrix is given below.
\begin{center}
\begin{tabular}{c||c|c|c|c|c|c}
$\cdot$ & $\psi_1$ & $\psi_2$ & $\psi_3$ & $\delta_{12}$ & $\delta_{13}$ & $\delta_{23}$ \\
\hline
\hline
$T_1$ & $2g$ & $1$ & $1$ & $1$ & $1$ & $0$ \\
$T_2$ & $1$ & $2g$ & $1$ & $1$ & $0$ & $1$ \\
$T_3$ & $1$ & $1$ & $2g$ & $0$ & $1$ & $1$ \\
$T_{12}$ & $4g+1 $& $4g+1$ & $2$ & $2g+2$ & $1$ & $1$ \\
$T_{13}$ & $4g+1$ & $2$ & $4g+1$ & $1$ & $2g+2$ & $1$ \\
$T_{23}$ & $2$ & $4g+1$ & $4g+1$ & $1$ & $1$ & $2g+2$
\end{tabular}
\end{center}
We perform the column operations that subtract $\sum_{j \neq i}\delta_{ij}$ from the $\psi_i$ column. Then, we perform the row operations that subtract $T_i + T_j$ from $T_{ij}$. When $n = 3$, this gives:
\begin{center}
\begin{tabular}{c||c|c|c|c|c|c}
$\cdot$ & $\psi_1 - \delta_{12} - \delta_{13}$ & $\psi_2 - \delta_{12} - \delta_{23}$ & $\psi_3 - \delta_{13} -\delta_{23}$ & $\delta_{12}$ & $\delta_{13}$ & $\delta_{23}$ \\
\hline
\hline
$T_1$ & $2g-2$ & $0$ & $0$ & $1$ & $1$ & $0$ \\
$T_2$ & $0$ & $2g-2$ & $0$ & $1$ & $0$ & $1$ \\
$T_3$ & $0$ & $0$ & $2g-2$ & $0$ & $1$ & $1$ \\
$T_{12} - T_1 - T_2$ & $0$& $0$ & $0$ & $2g$ & $0$ & $0$ \\
$T_{13} - T_1 - T_3$ & $0$ & $0$ & $0$ & $0$ & $2g$ & $0$ \\
$T_{23} - T_2 - T_3$ & $0$ & $0$ & $0$ & $0$ & $0$ & $2g$
\end{tabular}
\end{center}
More generally, using \eqref{i1}--\eqref{i4}, 
one can see that, using this change of basis, the intersection matrix always takes the block form
\begin{center}
\begin{tabular}{c||c|c}
    $\cdot$ & $\psi_i - \sum \delta_{ij}$ & $\delta_{ij}$ \\
    \hline\hline
    $T_i$ & $(2g - 2) \cdot \mathrm{Id}$ & $*$ \\
    $T_{ij} - T_i - T_j$ & $0$ & $2g \cdot \mathrm{Id}$
\end{tabular}
\end{center}
In particular, the intersection matrix is full rank for all $g \geq 2$ and any $n$. It follows that $\psi_1, \ldots, \psi_n$ and the boundary divisors $\delta_{ij}$ are independent on $\tilde{\H}_{g,n}$.
\end{proof}

\subsection{Spanning of $\psi$ classes} \label{spanning}
Our goal in this section is to show that $A^1(\H_{g,n})$ is spanned by $\psi_1, \ldots, \psi_n$ when $n \leq 2g+6$. The first step is the following.

\begin{lem} \label{no-extra}
For $n \leq 2g+6$, we have $A^1(\I_{g,n}^{\circ,1} \smallsetminus \Delta) = 0$.
\end{lem}
\begin{proof}
We described $\I_{g,n}^{\circ,1}$ in Section \ref{smn} as an open subset of a projective bundle $\pp \F$ over $\B^\circ \subseteq \B_{g,n}$, which in turn is an open subset of an affine bundle over $\BPU \times \BPU$.
By the projective bundle theorem $A^1(\pp \F)$ is generated by $\zeta := c_1(\O_{\pp \F}(1))$ and the generators $c_1 := c_1(\V)$ and $d_1 := c_1(\W)$ of $A^1(\BPU \times \BPU)$. Our goal is thus to show that $\zeta, c_1, d_1$ all restrict to zero on $\I_{g,n}^{\circ,1} \smallsetminus \Delta \subset \pp \F$. 

By construction, the pullback of $\V$ to $\I_{g,n}^{\circ,1} \subset \pp \F \to \BPU \times \BPU$ is the rank two bundle  $f_*\O_{\C}(p_1 + \overline{p}_1)$, whose projectivization is the universal $\pp^1$ bundle over $\I_{g,1}$. Therefore, $c_1$ is the pullback of $c_1$ from Lemmas \ref{1pt} and \ref{0pt} in Section \ref{relsec}. In particular, Equation \ref{dclass} says $\delta$ is a non-zero multiple of $c_1$, so $c_1 = 0 \in A^1(\I_{g,n}^{\circ,1} \smallsetminus \Delta)$.
Below, we find two more independent relations from studying components of the complement of $\I_{g,n}^{\circ,1} \subset \pp \F$. Let $e = g-n+1$ if $n \leq g$ and let $e = 0$ if $n \geq g+1$ so that
\[\F = \ker\left(\phi^* \E \to \sigma_1^*P^e_{\pp \V \times \pp \W/\pp \W}(\N) \oplus \bigoplus_{2 \leq j \leq n}\sigma_j \N\right).\]

\subsubsection{Meeting the horizontal ruling with too high multiplicity at $\sigma_1$} \label{firstrel}
Let $\F' \subset \F$ be the kernel of
\[\phi^* \E \to \sigma_1^*P^{e+1}_{\pp \V \times \pp \W/\pp \W}(\N) \oplus \bigoplus_{\substack{2 \leq j \leq n \\ j \neq g+1}}\sigma_j \N\]
Equations in $\pp \F'$ vanish to order $e+2$ along $\pi_2^{-1}(\pi_2(\sigma_1))$.
If $n \leq g$, then taking into account vanishing along the other sections, the vanishing of this equation meets the ruling with multiplicity $g+2$, so  contains the entire ruling. On the other hand, if $n \geq g+1$, then the vanishing of such an equation cannot be $g+1$ distinct points.
In either case, $\pp \F' \subset \pp \F$ lies in the complement of $\I_{g,n}^{\circ,1} \subset \pp \F$. 
By the snake lemma, we have
\[\F/\F' \cong \ker(\sigma_1^*P^{e+1}_{\pp \V \times \pp \W/\pp \W}(\N) \to \sigma_1^*P^{e}_{\pp \V \times \pp \W/\pp \W}(\N)) \cong \sigma_1^*(\N \otimes \Omega_{\pp \V \times \pp \W/\pp \W}^{\otimes e+1}).\]
The divisor $\pp\F' \subset \pp \F$ is defined by the vanishing of the composition
\[\O_{\pp \F}(-1) \to \F \to \F/\F' \cong \sigma_1^*(\N \otimes \Omega_{\pp \V \times \pp \W/\pp \W}^{\otimes e+1}).
\]
Therefore, the fundamental class of $\pp \F' \subset \pp \F$ is
\begin{align*}c_1(\O_{\pp \F}(1) \otimes \sigma_1^*(\N \otimes \Omega_{\pp \V \times \pp \W/\pp \W}^{\otimes g} )) &= \zeta + c_1(\sigma_1^*\pi_1^*\O_{\pp \V}(g+1)) \\
& \qquad + c_1(\sigma_1^*\pi_2^*\O_{\pp \W}(2)) + c_1(\sigma_1^*\Omega_{\pp \V \times \pp \W/\pp \W}^{\otimes e+1
}).
\intertext{First, note that $\Omega_{\pp \V\times \pp \W/\pp\W} \cong \pi_1^*\O_{\pp \V}(-2)$. 
Now, $v = \pi_1\circ \sigma_1$ represents the class $\O_{\pp \V}(1)$, so $v^*\O_{\pp \V}(1) = \pi_{1*}(c_1(\O_{\pp \V}(1))^2) = -c_1(\V) = -c_1$.
Similarly, $\sigma_1^*\pi_2^*\O_{\pp \W}(1) = -c_1(\W) = -d_1$. Therefore, the above becomes}
[\pp \F'] &= \zeta - (g+1)c_1  - 2d_1 + 2(e+1)c_1 \\
&= \zeta + (2e-g+1)c_1 - 2d_1.
\end{align*}
Hence, we obtain the relation $\zeta + (2e-g+1)c_1 - 2d_1 = 0 \in A^1(\I_{g,n}^{\circ,1})$.

\subsubsection{Meeting the vertical line with too high multiplicity at $\sigma_1$} \label{sp}
In $\I_{g,n}^{\circ, 1}$, the first marked point is not Weierstrass. Therefore, equations in $\pp \F$ that are tangent to the vertical ruling at $\sigma_1$ lie in the complement of $\I_{g,n}^{\circ, 1}$.
Let $\F'' \subset \F$ be the subbundle of equations whose vanishing is tangent to the vertical ruling at $\sigma_1$. Then 
\[\F/\F'' \cong \ker(\sigma_1^*P^1_{\pp \V \times \pp \W/\pp \V}(\N) \to \sigma_1^*\N) \cong \sigma_1^*(\N \otimes \Omega_{\pp \V \times \pp \W/\pp \V}).\]
 Similarly to before, we calculate
 \begin{align*}c_1(\sigma_1^*(\N \otimes \Omega_{\pp \V\times \pp \W/\pp \V})) &= c_1(\sigma_1^*\pi_1^*\O_{\pp \V}(g+1)) + c_1(\pi_1^*\pi_2^*\O_{\pp \W}(2)) + c_1(\sigma_1^*\Omega_{\pp \V \times \pp \W/\pp \V}) \\
 &= -(g+1)c_1 - 2d_1 + 2d_1 = -(g+1)c_1.
 \end{align*}
 Hence, the fundamental class of $\pp \F'' \subset \pp \F$ is $\zeta - (g+1)c_1$. Thus, we obtain a relation $\zeta - (g+1) c_1 = 0 \in A^1(\I_{g,n}^{\circ, 1} \smallsetminus \Delta)$. Since we have already said $c_1 = 0$, this shows $\zeta = 0$. Then the relation in \ref{firstrel} shows $d_1 = 0$ too.
\end{proof}

First, we conclude the case $n = 2$.
\begin{lem}
The classes $\psi_1, \psi_2$ form a basis for $A^1(\H_{g,2})$. Hence $A^1(\I_{g,2})$ is spanned by $\psi_1, \psi_2,$ and $\delta$.
\end{lem}
\begin{proof}
We have $\H_{g,2} \smallsetminus (D_{11} \cup D_{12}) = \I_{g,2}^{\circ, 1} \smallsetminus \Delta$. Hence $A^1(\H_{g,2} \smallsetminus (D_{11} \cup D_{12})) = 0$ by Lemma \ref{no-extra}. It follows that $\dim A^1(\H_{g,2}) \leq 2$. Meanwhile, Lemma \ref{psi-indep} shows $\psi_1$ and $\psi_2$ are independent in $A^1(\H_{g,2})$, 
so we can conclude that $\psi_1, \psi_2$ are a basis for $A^1(\H_{g,2})$. 
\end{proof}

This allows us to see that the divisors we have removed are tautological for any $n$.
\begin{lem} \label{dijl}
For any $g, n$, the class $[D_{ij}] \in A^1(\I_{g,n})$ lies in the span of $\psi_i, \psi_j$ and $\delta$.
\end{lem}
\begin{proof}
If $i \neq j$, the divisor $D_{ij}$ is the pullback of $D_{12} \subset \I_{g,2}$ under the map $\I_{g,n} \to \I_{g,2}$ that forgets all but the $i^{\mathrm{th}}$ and $j^{\mathrm{th}}$ marked points. Then this follows from the previous lemma which said $\psi_1, \psi_2, \delta$ span $A^1(\I_{g,2})$.

If $i = j$, then $D_{ii}$ is the pullback of $D_{11} \subset \I_{g,1}$, so this follows from Lemma \ref{1pt} which showed $A^1(\I_{g,1})$ is spanned by $\psi_1$ and $\delta$.
\end{proof}

\begin{lem}
If $n \leq 2g+6$, then
the classes $\psi_1, \ldots, \psi_n$ form a basis for $A^1(\H_{g,n})$. Hence, $A^1(\I_{g,n})$ is spanned by $\psi_1, \ldots, \psi_n$ and $\delta$.
\end{lem}
\begin{proof}
By Lemma \ref{no-extra}, we see $A^1(\H_{g,n} \smallsetminus \bigcup_{i,j} D_{ij}) = 0$. But by Lemma \ref{dijl}, each $[D_{ij}]$ lies in the span of $\psi_i$ and $\psi_j$. It follows that $A^1(\H_{g,n})$ is spanned by $\psi_1, \ldots, \psi_n$. They are a basis by Lemma \ref{psi-indep}.
\end{proof}

Having established (a) and (b) from the beginning of this section, our previous proof of Theorem \ref{divgen} now gives the result in any characteristic $\neq 2$.

\bibliographystyle{amsplain}
\bibliography{refs}
\end{document}